\documentclass[11pt]{article}

\usepackage{graphicx}
\usepackage{latexsym}
\usepackage{amssymb}
\usepackage{amsmath}
\usepackage{enumerate}

\usepackage{amsmath}
\usepackage{stmaryrd}
\usepackage{hyperref}
\usepackage{amssymb}
\usepackage{CJK}
\usepackage{color}

\usepackage{amsmath,latexsym,amssymb,amsthm,array,amsfonts}
\usepackage{mathtools, stmaryrd}
\usepackage{xparse} \DeclarePairedDelimiterX{\Iintv}[1]{\llbracket}{\rrbracket}{\iintvargs{#1}}
\NewDocumentCommand{\iintvargs}{>{\SplitArgument{1}{,}}m}
{\iintvargsaux#1} %
\NewDocumentCommand{\iintvargsaux}{mm} {#1\mkern1.5mu..\mkern1.5mu#2}
\usepackage{mathrsfs,dsfont,bbm}
\usepackage{csquotes}
\usepackage{layout}
\usepackage{layout}

\newtheorem{lemma}{Lemma}
\newtheorem{definition}{Definition}
\newtheorem{corollary}{Corollary}
\newtheorem{theorem}{Theorem}
\newtheorem{proposition}{Proposition}
\newtheorem{remark}{Remark}

\def\real{{\mathord{{\rm I\kern-2.8pt R}}}}        
\def\inte{{\mathord{{\rm I\kern-2.8pt N}}}}

\def\sZZ{{\rm Z\kern-2.8ptem{}Z}}

\def\z{{\mathchoice
  {\sZZ}
  {\sZZ}
  {\rm Z\kern-0.30em{}Z}
  {\rm Z\kern-0.25em{}Z} }}
\def\sQQ{{\kern 0.27em \vrule height1.45ex width0.03em depth0em
          \kern-0.30em \rm Q}}
\def\qu{{\mathchoice
    {\sQQ}
    {\sQQ}
  {\kern 0.225em \vrule height1.05ex width0.025em depth0em \kern-0.25em \rm Q}
  {\kern 0.180em \vrule height0.78ex width0.020em depth0em \kern-0.20em \rm Q}
        }}
\def\sCC{{\kern 0.27em \vrule height1.45ex width0.03em depth0em
          \kern-0.30em \rm C}}
\def\complex{{\mathchoice
    {\sCC}
    {\sCC}
  {\kern 0.225em \vrule height1.05ex width0.025em depth0em \kern-0.25em \rm C}
  {\kern 0.180em \vrule height0.78ex width0.020em depth0em \kern-0.20em \rm C}
        }}

\MakeOuterQuote{"}

\newcommand{\ba}{\begin{array}}
\newcommand{\ea}{\end{array}}
\newcommand{\be}{\begin{equation}}
\newcommand{\ee}{\end{equation}}
\newcommand{\bea}{\begin{eqnarray}}
\newcommand{\eea}{\end{eqnarray}}
\newcommand{\beaa}{\begin{eqnarray*}}
\newcommand{\eeaa}{\end{eqnarray*}}

\def\g{\Lambda}

\def\z{\zeta}

\font\tenmath=msbm10 \font\sevenmath=msbm7 \font\fivemath=msbm5
\newfam\mathfam \textfont\mathfam=\tenmath
\scriptfont\mathfam=\sevenmath \scriptscriptfont\mathfam=\fivemath

\def \={{\buildrel {\rm (law)} \over =}}

\def\qed{ \hfill \vrule width.25cm height.25cm depth0cm\smallskip}

\newcommand{\basa}{\begin{assumption}}
\newcommand{\easa}{\end{assumption}}

\newcommand{\bas}{\begin{assum}}
\newcommand{\eas}{\end{assum}}

\newcommand{\E}{\mathbb{E}}

\newcommand{\ignore}[1]{}
\usepackage{authblk}
\begin{document}

\renewcommand{\thefootnote}{\fnsymbol{footnote}}

\renewcommand{\thefootnote}{\fnsymbol{footnote}}

\title{Free Stein kernels, free moment maps, and higher order derivatives}

\author[1]{Charles-Philippe Diez \thanks{charlesphilippemanuelf.diez@cuhk.edu.hk}}
\affil[1]{Department of Statistics, The Chinese University of Hong-Kong.}
\renewcommand\Authands{ and }

\maketitle

\begin{abstract}
In this work, we describe new constructions of free Stein kernels. Firstly, in dimension one, we propose a free analog to the construction of Stein kernels using moment maps as the one proposed by Fathi in \cite{mm}. This will be possible for a class of measures called the "{\it free moment measures}" via the notion of {\it free moment map} (convex functions), introduced in the free case by Bahr and Boschert in \cite{BB}.
In a second time, we introduce the notion of higher-order free Stein kernels relative to a potential, which can be thought as the free counterpart of a recent and powerful idea introduced in the classical case by Fathi \cite{F}, and which generalize the notion of free Stein kernels by introducing higher-order derivatives of test functions (in our context noncommutative polynomials). We then focus our attention to the case of homothetic semicircular potentials. We prove as in the classical case, that their existence implies moments constraints. Finally, we relate these discrepancies to various metrics: the free (quadratic) Wasserstein distance, the relative free Fisher information along the Ornstein-Uhlenbeck flow or the relative non-microstates free entropy. Finally, as an important application, we provide new rates of convergences in the entropic {\it free CLT} under higher moments constraints.
\end{abstract}

\maketitle

\section{Introduction}
Stein's method, invented by Charles Stein in 1972 is a powerful tool to obtain bounds for distances between probability measures, and to obtain rate of convergence in the central limit theorem for various metrics. Numerous applications of Stein's method relies on its connection with various branchs of mathematics: exchangeable pairs, Malliavin calculus in gaussian or poissonian settings, or more generally by a Markov triple approach. Transportation cost inequalities and relations between quantities such as $L^p,p\geq 1$-Wasserstein distances and the entropy or relative Fisher information were established using various branches of Mathematics: PDE techniques (Otto and Villani \cite{VI}), semigroup approach or Gamma-calculus (see for example the monograph of Bakry, Gentil Ledoux \cite{BGL} to have a complete exposure of the theory).
Voiculescu in a series of breakthrough papers \cite{Vcg, Voic2, V} invented the concept of free probability and discovered an analog of free information theory: free Fisher information, microstates and non-microstates free entropy, non-commutative Hilbert transform\ldots, and has shown they behave as well as in the commutative case.
   Shlyakhtenko in \cite{SCLT} has proved that the free entropy (we remind that in dimension one both version coincides) is monotone along the free central limit theorem.
        Numerous functional inequalities, which holds in the commutative case have been proved to hold true in our context: {\it Free Stam inequality, Cramer-Rao bound, log-Sobolev inequality} and more recently the {\it free HSI inequality}. We will focus on the last one whose proof relies on the study of free Stein kernels. In fact, surprisingly and contrary to the classical case, the existence of free Stein kernels relative to a potential is always assured provided a moment condition relative to the cyclic derivative potential is fulfilled, see \cite{FCM} for precise statements. A deep connection with the free Malliavin calculus has been developed through the years, since the first striking contribution about "{\it fourth moments theorems on the Wigner space}" by Kemp, Nourdin, Peccati, Speicher in \cite{KNPS}, and which have been recently greatly improved by Bourguin and Campese in \cite{BC}, Cébron \cite{C}, and recently extended to the multi-dimensional setting by the author in \cite{Di}.
\bigbreak
Fathi in his important paper \cite{mm} discovered a new way to implement Stein kernel with respect to the standard Gaussian measure via ideas of optimal transportation. Indeed, there is a variant variant of the Monge-Kantorovitch problem where we only fix a target measure $\mu$ and we are looking for a convex, essentially continuous function $\phi$ called the "{\it moment map}" as well as its associated Gibbs measure $e^{-\phi}dx$ such as $\mu$ is the pushforward of $e^{-\phi}$ by the gradient of $\phi$, i.e $\mu=(\nabla \phi)_{\sharp} e^{-\phi}dx$. This formal construction can thus be seen as a bijection between convex functions and measures. Existence of such functions $\phi$ was was settled in the affirmative by Cordero-Erausquin and Klartag in \cite{CEK} for a broad class of probability measures $\mu$ called "{\it moment measures}" which have barycenter at the origin, are not supported on a hyperplane, and with finite first moment. Fathi realised that it was possible to construct Stein kernels with the help of these {\it moment maps}, and thus, this gives a new powerful tool to compare $\mu$ to the standard Gaussian measure $\gamma$. He was also able to generalize his construction of Stein kernel with respect to more general reference measures, i.e for $e^{-V}dx$, with $Hess\: V>0$, and was able to obtain new (and also sharp with respect to the dimension) estimates for the CLT in $L^p, p\geq 2$ Wasserstein distances for uniformly log-concave isotropic probability measures.
\bigbreak
    Recently Fathi, discovered by looking at the edgeworth expansions in the classical {\it CLT}, that we should expect better rate of convergence when the random variables have their moments of order $k\geq 2$ which match those of the standard gaussian. These edgeworth expansions carries naturally the integration-by-parts involving the higher-order Hermite polynomials (the Stein identity which correspond to the case $k=1$ for which the associated Hermite polynomial is $H_1=X$). Following this idea, Fathi in \cite{F} generalized (in two directions) the notion of Stein kernels: matrix valued measurable map relatively to the classical Gaussian potential, by introducing higher-order Stein kernels (symmetric tensors valued mapping) involving higher-order derivatives of test functions, and have provided their existence when the random vector of interest satisfies a {\it Poincaré inequality}. He also shown new important types of {\it HSI inequalities}, new {\it Wasserstein-Stein discrepancies} inequalities, new bounds for the $L^2$-Wasserstein distance in the central limit theorem, as well as for its entropic version. We also note that he was able to show that the {\it Zolotarev distances} to the standard gaussian measure (a class of distances which generalise the usual $1$-Wasserstein, and which are defined via duality over smooth functions with their $k$-order derivative bounded by $1$), are controlled by the relative Stein discrepancy (see section 2.2 in \cite{F}). It would also be of great interest to define and study these metrics in the free context (firstly in the scalar case, and only for now in the one dimensional case, as it is not clear for now how to introduce analogues of classical distances in the multi-dimensional setting, which are not defined via couplings. It would also be very interesting to investigate the operator-valued free probabilistic setting) as for now many quantitative estimates for the {\it free central limit theorem} are only given for the quadratic Wasserstein, or in terms of {\it Cauchy-Transform}. In the free context, a similar phenomenon was investigated by Chiastyakov and G\"otze in \cite{edge} which have provided deep results by using approximations of self-normalized sums of free random variables by free Meixner laws.
\bigbreak
The purpose of this article is to generalise the results of Fathi \cite{mm, F}, and Fathi and Nelson \cite{FN}, in a free probabilistic context. Note that even if the notion of free Stein kernel is firsly a probabilistic one, this work is also motivated by the recent advances of Charlesworth and Nelson in \cite{CN,CN2} which have introduced the notion of {\it free Stein dimension}, which is an invariant for finitely generated unital, tracial $*$-algebras, and which is deeply connected to the notions of free Fisher information, free entropy, and the free entropy dimension. This new quantity thus enlarges the tools of interest to prove structural properties of finitely generated tracial von Neumann algebras (absence of property $\Gamma$, non-$L^2$ rigidity, absence of Cartan subalgebras...).

\section{Definitions and notations}\label{sec1}
\bigbreak
Let's denote $\mathcal{M}$ a von-Neumann algebra equipped with $\tau$ a faithful normal tracial state. Let $\mathbb{P} =\mathbb{C}\langle t_1,\ldots,t_n \rangle=\mathbb{C}\langle T \rangle$ be the $*$-algebra of noncommutative polynomials in $n$ self-adjoints: $t_i=t_i^*$ non-commuting variables $T=(t_1,\ldots,t_n)$ and equipped with the involution $(PQ)^*=Q^*P^*$.

\bigbreak
\begin{definition}
A free Stein kernel for a $n$-tuple $X=(x_1,\ldots,x_n)\in (\mathcal{M},\tau)^n$ with respect to a potential $V\in \mathbb{P}^{(R)}$ is an element of $L^2(M_n(\mathcal{M}\otimes \mathcal{M}^{op}),(\tau \otimes \tau^{op}) \circ Tr))$ such that for any $P \in \mathbb{P} ^n$
\begin{equation}
    \langle [\mathcal{D}V](X),P(X)\rangle_{\tau}=\langle A,[\mathcal{J}P](X)\rangle_{\tau \otimes \tau^{op}},
\end{equation}
\end{definition}
The Stein discrepancy of $X$ relative to $V$ is then defined as 
\begin{equation}
\Sigma ^*(X|V)=\inf_{A}\lVert A-(1\otimes 1^{op})\otimes I_n\rVert_{L^2(M_n(\mathcal{M}\otimes \mathcal{M}^{op}),(\tau \otimes \tau^{op}) \circ Tr))},
\end{equation}
where the infinimum is taken over all admissible Stein kernel $A$ of $X$ relative to $V$.

\begin{flushleft}
Here $\mathcal{D}V$ is a cyclic gradient for which , and $\mathcal{J}P$ is the Jacobian matrix of $P=(P_1,\ldots,P_n)$ that we will define in the next pages.
\end{flushleft}
\begin{flushleft}
In dimension one, thanks to the the positivity of the state and the Riesz-Markov theorem, the {\it law or analytic distribution} of a self-adjoint noncommutative random variable $X\in \mathcal{M}_{s.a}$ (and thus bounded), is the unique Borel compactly supported probability measure $\mu_X$ such that:
\begin{eqnarray}
\tau(X^n)=\int x^nd\mu_X(x),
\end{eqnarray}
for all $n\geq 0$.
\end{flushleft}
\begin{flushleft}
    More generally we call the {\it law} of $X=(x_1,\ldots,x_n)\in \mathcal{M}_{s.a}^n$, the linear map $\mu_{x_1,\ldots,x_n}$ defined as:
    \begin{eqnarray}
    \mathbb{P} &\rightarrow& \mathbb{C}\nonumber\\
    P &\mapsto& \tau(P(x_1,\ldots,x_n)),
    \end{eqnarray}
\end{flushleft}
\begin{definition}
The centered semicircular distribution of variance $\sigma^2$ is the probability distribution :
\begin{equation}
    S(0,\sigma^2)(dx)=\frac{1}{2\pi\sigma^2}\sqrt{4\sigma^2-x^2}dx, \; \lvert x\rvert\leq 2\sigma,
\end{equation}
This distribution has all his odd moments which vanish by the symmetry of the distribution around $0$. Its even moments can be expressed by the help of the {\it Catalan numbers} through the following relation valid for all non-negative integers $m$ :
\begin{equation}
    \int_{-2\sigma}^{2\sigma}x^{2m}S(0,\sigma^2)(dx)=C_m\sigma^{2m},
\end{equation}
Where $C_m=\frac{1}{m+1}\binom{2m}{m}$ is the {$m$-th Catalan number}.
\end{definition}
  For a polynomial $p \in \mathbb{P}$ and a monomial $m$, we let $c_{m(p)}\in \mathbb{C}$ which denotes the coefficient of $m$ in $p$. And as in \cite{FN}, for each $R > 0$, we define the following norm:
\begin{equation}
    \lVert p\rVert_R=\sum_m c_m(p)R^{deg (m)},
\end{equation}
where the (finite) sum is running over all monomials appearing in "$p$". We can then take the completion of $\mathbb{P}$ with respect to this norm and we will denote the space as $\mathbb{P}^{(R)}$, which can be thought as the formal power series of radius of convergence at least "$R$".
\begin{flushleft}
For $V \in \mathbb{P}^{(R)}, R > \lVert X\rVert$, we say that the joint law of $X$ with respect to $\tau$ is a {\it free Gibbs state}  with potential $V$ if for each $j = 1,\ldots,n$ and each $p \in \mathbb{P}$ :
\begin{equation}
    \langle [\mathcal{D}_jV](X),p(X)\rangle_2 =\langle 1\otimes 1^{op},[\partial_jp](X)\rangle_{HS},
\end{equation}
That is, if the conjugate variables to $X$ are given by $[D_1V ](X),\ldots, [D_nV](X)$. Equivalently, the following equation holds for all $P \in \mathbb{P}^n$:
\begin{equation}
    \langle[\mathcal{D}V](X),P(X)\rangle_2 =\langle (1\otimes 1^{op})\otimes I_n,[\mathcal{J}P](X)\rangle_{HS},
\end{equation}
\end{flushleft}
\begin{flushleft}
As in the classical case, Free Gibbs laws are characterized by a minimization problem: they minimize the relative microstates variant of the entropy:
\end{flushleft}
\begin{definition}(\cite{GM,FN,VF})
The free Gibbs law $\tau_V$ associated to the potential $V$ if it exists, is the minimizer of the functional:
\begin{equation}
-\chi(\tau)+\tau(V),
\end{equation}
\end{definition}
In particular, and thanks to Voiculescu explicit computations, we recall that in the one dimensional case the free entropy is the logarithmic energy (up to an explicit constant). Thus, we are left to minimize the following functional.
\begin{definition}
The free Gibbs measure $\nu_u$ associated with the convex function $u:\mathbb{R}\rightarrow \mathbb{R}$ is the measure corresponding to the free Gibbs law $\tau_u$. That is the unique minimizer of the functional :
\begin{equation}
-\int\int \log\lvert t-s\rvert d\mu(t)d\mu(s)+\int u(s)d\mu(s),
\end{equation}
\end{definition}
\begin{remark}
Assume also that the free Gibbs measure $\nu_u$ exists (and might be possibly unique), e.g when $u$ is strictly convex on a sufficiently large interval, that is when it exits $\kappa>0$, such that $u''(x)\geq \kappa $ for all $x\in \mathbb{R}$ or that $u$ is bounded below, satsifies some growth condition and a Holder-continuous type criterion \cite{DmV}. The other case where existence and uniqueness is ensured is when we consider "small" analytic perturbation of the semicircular potential, i.e $\frac{1}{2}x^2+\beta W$ (see, e.g \cite{DmV}). The more powerful result (quartic potentials could be considered) about existence and uniqueness can be found in the breakthrough paper of Dabrowski, Guionnet and Shlyakhtenko \cite{YGS}. In a non-convex situation, the problem become much harder to understand. However, it is interesting to notice that in this setting Maida and Maurel-Segala \cite{maida} were able to prove free $T_1(C)$ (Wasserstein-$1$) transport inequalities.
\bigbreak
We recall that in dimension one, we have:
$\mathcal{J}f(x,y)=\frac{f(x)-f(y)}{x-y}$ and that the cyclic derivative coincide with the usual derivative, i.e $\mathcal{D}u(x)=u'(x)$.
\newline
    Then the previous free Stein equation for the potential $u$ for $\mu$ reformulated in its {\it "measure" (commutative) version} reduces to the existence of a function 
$A:\mathbb{R}^2\rightarrow \mathbb{R}$, such that for all nice test functions $f$:
\begin{equation}
\int u'(x)f(x)d\mu(x)=\int\int A(x,y)\frac{f(x)-f(y)}{x-y}d\mu(x)d\mu(y),
\end{equation}
\end{remark}

\begin{flushleft}

Recall by the GNS construction, $\tau$ defines an inner product on $\mathcal{M}$ by setting for all $x,y \in \mathcal{M}$
\begin{equation}
    \langle x,y\rangle_{\tau}=\tau (y^*x),
\end{equation}
\end{flushleft}
\begin{flushleft}
The completion of $\mathcal{M}$ with respect to the induced norm
$\lVert.\rVert_{\tau}$ is denoted $L^2(\mathcal{M},\tau)$. We will omit to denote the trace when its clearly defined and denote $\lVert.\rVert_{\tau}$  as $\lVert.\rVert_{2}$ and $L^2(\mathcal{M},\tau)$ as $L^2(\mathcal{M})$.
\end{flushleft}
From the von Neumann tensor product $\mathcal{M}\bar\otimes\mathcal{M}^{op}$ equipped with the faithful normal state $\tau\otimes \tau^{op}$, we can consider the Hilbert space $L^2(\mathcal{M}\bar\otimes \mathcal{M}^{op},\tau\otimes \tau^{op})$. This space can be identified with $HS(L^2(\mathcal{M}))$: the two-sided ideal of Hilbert–Schmidt operators on $L^2(M)$ via the isometric extension of the following map
\begin{equation}
    x\otimes y \mapsto \langle y,.\rangle_2 x,
\end{equation}
We also denote the inner product 
\begin{equation}
\langle a\otimes b, c\otimes d\rangle_{\tau\otimes \tau}=\tau\otimes\tau(c^*\otimes d^*\sharp (a\otimes b)):=\tau(c^*a)\tau(bd^*)
\end{equation}
\begin{flushleft}
For the $n$-tuple $X$, we define $\lVert X\rVert=\max_j \lVert x_j\rVert$. 
\newline
We will also write $C^*(X)$ and $W^*(X)$ for the $C^*$-algebra and von Neumann algebra generated by $x_1,\ldots, x_n$, respectively.
\newline
We will also denote $Var(X)$ as $\sum_{j=1}^n \tau(x_j^*x_j)$.
\newline
Given another $n$-tuple $Y = (y_1,\ldots,y_n) \in \mathcal{M}^n$, we write
$\langle X,Y\rangle_{\tau} =\sum_{j=1}^n\langle x_j,y_j\rangle_{\tau}$.
\end{flushleft}
\begin{definition}(Voiculescu \cite{Vcg})
Formally the cyclic derivative is defined on monomials $m\in \mathbb{P}$ as :
\begin{equation}
    \mathcal{D}m=(\mathcal{D}_1m,\ldots,\mathcal{D}_nm),
\end{equation}
where
\begin{equation}
\mathcal{D}_jm=\sum_{m=at_{j}b}ba,
\end{equation}
and then extended linearly to $\mathbb{P}$.
\end{definition}
\begin{definition}
(Voiculescu \cite{V})
The j-free difference quotient is defined as :
\begin{equation}
    \partial_jm=\sum_{m=at_jb}a\otimes b^{op},
\end{equation}
and then extended linearly to $\mathbb{P}$.

\end{definition}

\begin{definition}(Voiculescu \cite{V})
For $P=(p_1,\ldots,p_n)$, we define the Jacobian as :
\begin{equation*}
\mathcal{J}P = 
\begin{pmatrix}
\partial_1 p_1 & \partial_2 p_1 & \cdots & \partial_n p_n\\
\vdots  & \vdots  & \ddots & \vdots  \\
\partial_1 p_n & \partial_2 p_n & \cdots & \partial_n p_n
\end{pmatrix} \in M_n(\mathbb{P}\otimes \mathbb{P}^{op}),
\end{equation*}
\end{definition}
\begin{remark}
    We also recall that we will often write and view for a tuple of noncommutative random variables $X=(x_1,\ldots,x_n)\in \mathcal{M}^n$, both $\partial_X=(\partial_1,\ldots,\partial_n)$ and $\mathcal{J}_X$ viewed as acting respectively onto
    \newline
    $\partial_i:\mathbb{C}\langle x_1,\ldots,x_n\rangle\rightarrow L^2(\mathcal{M})\otimes L^2(\mathcal{M}^{op})\simeq HS(L^2(\mathcal{M}))$ and $\mathcal{J}_X: \mathbb{C}\langle x_1,\ldots,x_n\rangle\mapsto L^2(M_n(\mathcal{M})\otimes L^2(\mathcal{M}^{op}))$
   trough the canonical evaluation homomorphism: $(ev_{X}\otimes ev_{X})$ and $(ev_X\otimes ev_X)\otimes I_n$.

\end{remark}
We recall here various definitions which are the central point of interests and which constitutes the non-microstates counterpart of the free information theory. 
\begin{definition}
 We say that $X$ admits conjugate variables $\xi_{x_1},\ldots,\xi_{x_n}\in L^2(W^*(X))$ if the following relation holds true for every $p\in \mathbb{P}$ :
\begin{equation}
    \langle \xi_{x_j},p(x)\rangle_2=\langle 1\otimes 1^{op},[\partial_jp](X)\rangle_{HS},
\end{equation}
Moreover we define higher-order (mixed) conjugate variables, here of order $k\geq 1$ associated to $j_1,\ldots,j_k$, if it exists $\xi_{j_1\ldots,j_k}\in L^2(W^*(X))$:
\begin{equation}
    \langle \xi_{j_1\ldots,j_k}^k,p(x)\rangle_2=\langle 1^{\otimes (k+1)} ,[\partial_{j_1,\ldots,j_k} ^kp](X)\rangle_2,
\end{equation}
where $\partial_{(j_1,\ldots,j_k)} ^k$ denote the k-order free difference quotient with respect to $x_{j_1},\ldots,x_{j_k}$.
\bigbreak
The free Fisher information of X is defined as 
\begin{equation}
    \Phi^*(X):=\sum_{i=1}^n \lVert \xi_{x_j}\rVert_2 ^2,
\end{equation}
if conjugate variables for $X$ exists and $+\infty$ otherwise.
\end{definition}
\begin{remark}
In this way,
we have that $\Phi^*(X) <\infty $ if and if only $(1\otimes 1^{op})\otimes I_n$ belongs to $dom(\mathcal{J}_X^*)$ which is in this case viewed as the adjoint of $\mathcal{J}$ a densely defined closable operator whose acts on $B\in M_n(\mathcal{M}\otimes \mathcal{M}^{op})$ as :
\begin{equation}
    \mathcal{J}^*(B)=\left(\sum_{i=1}^{n}\partial_{i}^*(B_{i,j})\right)_{j=1}^{n},
\end{equation}
\end{remark}
\begin{definition}
 The free Fisher information of X relative to a potential $V \in \mathbb{P}^{(R)} $ for $R > \lVert X\rVert $ is the quantity :
 \begin{equation}
     \Phi^*(X|V)=\sum_{j=1}^n \lVert \xi_{x_j}-[\mathcal{D}_jV](X)\rVert_2 ^2,
 \end{equation}
 if such conjugate variables for $X$ exists and + $\infty$ otherwise.
\end{definition}
In the sequel, we will focus our attention to the particular case of a subclass of semicircular potentials (the ones with positive homothetic covariance matrices). We will denote and consider $V_\rho$ the following potential defined for all $\rho > 0$ by :
\begin{equation}
V_\rho=\frac{\rho}{2} \sum_{j=1}^n t_j^2,
\end{equation}
\begin{definition}
 Relative non-microstates free entropy. 
 Let $S = (s_1,\ldots , s_n)$ be a free $(0, 1)$-semicircular $n$-tuple,
free from $X=(x_1, \ldots, x_n)$. Then the non-microstates free entropy of $x_1,\ldots, x_n$ is defined in \cite{V}  to be the quantity 
\begin{equation}
    \chi^*(x_1,\ldots,x_n)=\frac{1}{2}\int_{0}^{\infty}\left(\frac{n}{1+t}-\Phi^*(X+ \sqrt{t}S)\right)dt+\frac{n}{2}log(2\pi e),
\end{equation}
which we also denote by $\chi^*(X)$.
If $X$ is a free $(0, \rho^{-1})$-semicircular family, then $X +\sqrt{t}S$
therefore $\tau^*(X + \sqrt{t}S) = n(\rho^{-1}+ t)^{-1}$. From this it is easy to compute $\chi^*(X) = \frac{n}{2}log(2\pi e\rho^{-1})$.
\end{definition} 
\bigbreak
\begin{definition}
For $X = (x_1,\ldots,x_n) \in \mathcal{M}^n$, let $V \in \mathbb{P}^{(R)}$ for $R \geq \lVert X \rVert$, we define the non-microstates
free entropy of $x_1,\ldots, x_n$ relative to $V$ to be the quantity
\begin{equation}
\chi^*(x_1,\ldots,x_n |V):=\tau(V(X))-\chi^*(x_1,\ldots,x_n),
\end{equation}
\end{definition}
which we will also denote by $\chi^*(X | V )$. In this context, we refer to $V$ as a potential.
Since $\chi^*(.)$ is maximized (for fixed variance) by an n-tuple of free semicircular operators, it is easy to see
that $\chi^*(.| V_\rho)$ is minimized by a free $(0, \rho^{-1})$-semicircular $n$-tuple.
\begin{definition}(Connes-Shlyakhtenko def $4.1.1$ in \cite{ConnesS})
The non-microstates free entropy dimension (the entropy variant) is defined for $X=(X_1,\ldots,X_n)$ self-adjoints and $S=(S_1,\ldots,S_n)$ a standard semicircular system supposed free from $X$ in some $W^*$-tracial probability space $(\mathcal{M},\tau)$:
\begin{equation}
    \delta^*(X)=n-\underset{\epsilon\rightarrow 0}\liminf \:\frac{\chi^*(X+\sqrt{\epsilon}S)}{\frac{1}{2}\log \epsilon},
\end{equation}

\end{definition}
We also recall that in dimension one, both microstates and non-microsates entropy agree in full generality (see e.g Voiculescu \cite{V}) 
The following fundamental theorem was proved by Shlyakthenko in \cite{SCLT} and provides an important property: the monotonicity of the free entropy along the "{\it Central Limit Theorem}" which means the following: 

\begin{theorem}(Shlyakthenko in \cite{SCLT})
    For any sequence of self-adjoint, freely independent and identically distributed random variables $(y_i)_{i\in \mathbb{N}^*}$, then:
    \begin{equation}
    n\mapsto \chi\left(\frac{y_1+\ldots,y_n}{\sqrt{n}}\right),
\end{equation}
is monotone increasing for all $n\geq 1$.
\end{theorem}

\begin{flushleft}
\textbf{We recall to the reader that we will only work in the tracial case in the sequel.}
\end{flushleft}
\section{Free Stein kernels and free moments maps}
In this section focused only in the one-dimensional (commutative) case where we have a proper notion of the analytic distribution $\mu_X$ of a noncommutative random variable $X\in (\mathcal{M},\tau)$ (we will only focus on the "measure" version and deal with compactly supported measure instead of noncommutative laws), we will investigate the construction of free Stein kernels for compactly supported measure thanks to a notion of {\it free moment maps} developed in the free setting by Bahr and Boschert in \cite{BB}, and which corresponds to a free analog of a breakthrough result in the classical case by Cordero-Erausquin and Klartag in \cite{CEK}, and which gives a broad class of probability measure that are the pushforward of a convex function $\phi$ with respect to the law of density $e^{-\phi}$. Note that this is a variant of the optimal transportation problem where here we only fix a target measures and we are looking for a original measure as well as its transport map which turns out to be a gradient of the potential of the original Gibbs measure $e^{-\phi}dx$. 
\begin{flushleft}
In the sequel, we recall that for $\phi:\mathbb{R}^n\rightarrow \mathbb{R}\cup\left\{+\infty\right\}$ a convex function, we denote its Legendre transform:
\begin{equation}
\phi^*(y)=\underset{x\in \mathbb{R}^n}{\sup}\left\{x\cdot y-\phi(x)\right\},
\end{equation}
And which satisfies the following property:
\begin{enumerate}
\item if $\phi$ is $\mathcal{C}^2$, $\phi^*$ is also $\mathcal{C}^2$.
\item $\nabla \phi^*$ is the inverse of $\nabla \phi$, i.e: $\nabla \phi^*=(\nabla \phi)^{-1}$.
\end{enumerate}
\end{flushleft}
\begin{theorem} (Cordero-Erausquin and Klartag in \cite{CEK}) 
Let $\mu$ be a centered measure on $\mathbb{R}^n$, with finite first moment and not supported on an hyperplane. Then there exists a convex function $\phi$, such that $\mu$ is the pushforward of the centered probability measure with density $e^{-\phi}$, by the map $\nabla \phi$. The function $\phi$ is called the moment map of $\mu$.
\end{theorem}
The function $\phi$ might not be smooth in full generality, especially when the moment laws is a combination of Dirac masses. This fact was already noticed by Fathi (see remark preceding theorem 2.2 in \cite{mm}), and Bahr and Boschert (see e.g 2.5.2 in \cite{BB}). Indeed consider the uniform measure on $\left\{-1,1\right\}$ viewed as a subset of $\mathbb{R}$, that is $\mu=\frac{1}{2}\delta_{-1}+\frac{1}{2}\delta_1$, for which we can compute $\phi(x)=\frac{1}{2}\lvert x\rvert$ which is not smooth at the origin.

\begin{flushleft}
Bahr, Boschert have been able to prove a free analog of this construction, which is of course better adapted in the one dimensional (commutative) case. Note also that the condition that the centered measure $\mu$ is not supported on an hyperplane is reduced in this one dimensional case to the condition: $\mu\neq \delta_0$.
\end{flushleft}
\begin{theorem}(Bahr, Boschert Theorem 2.5 in \cite{BB})
Let $\mu\neq \delta_0$, a centered probability measure  in $\mathcal{P}_2(\mathbb{R})$, then for some convex function $u$, $(u')_{\sharp}\nu_{u}=\mu$, i.e $\mu$ is the pushforward of the free Gibbs measure $\nu_u$ by the function $u'$. The convex function $u$ is called the {\it free moment map} of $\mu$.
\end{theorem}
\begin{flushleft}
This theorem gives a broad class of probability measures which are the pushforward of free Gibss measure $\nu_u$ by the function $u'$, and in particular it provides a (complete, with only a little restriction) analog of the result of Cordero-Erausquin and Klartag in the classical case, but only in dimension one.
\end{flushleft}
More interestingly, and in the classical case, Fathi in his paper \cite{mm} has discovered a very interesting connection between the {\it moment map} $\phi$ and the notion of Stein kernel. Precisely, he remarked that the map:
\begin{equation}
\mu \mapsto e^{-\phi}dx,
\end{equation}
where (here $dx$ standard for the Lebesgue measure on $\mathbb{R}^n$ up to a renormalization constant) has a unique fixed point given by the standard Gaussian measure, for which we recall that the moment map is given by $\frac{\lVert x\rVert_2^2}{2}$. Thus, the moment map already contains a lot of information on how $\mu$ is close to the standard Gaussian measure $\gamma$.
\bigbreak
In particular, and heuristically, if $\phi(x)\approx \frac{\lVert x\rVert_2^2}{2}$, then we can expect that $\mu \approx \gamma$. Fathi formalize this argument, and has shown that one can construct free Stein kernel thanks to this moment map. It is then easy to bound the quadratic Wasserstein distance to the standard gaussian thanks to the {\it Wasserstein-Stein discrepancy} inequality of Ledoux, Nourdin and Peccati \cite{LNP}. Fathi also constructed Stein kernel with respect to possibly non-Gaussian reference measure, i.e: for classical Gibbs state $e^{-V}dx$, associated with a potential $V$ convex, $\mathcal{C}^2$ and $Hess V>0$, via the construction of a Stein kernel with respect to the standard gaussian for the pushforward $\mu_V$ of $\mu$ by $\nabla V$ (see, e.g, section 5 in \cite{mm}), and gives in this way a powerful tool to compare $\mu$ and a large class of references measures: the strictly log-concave measures associated with a smooth convex potential $V$, which could be far from $\gamma$.
\begin{flushleft}
This implies in particular from the convexity of $\phi$ that $\nabla \phi$ is the optimal transportation map between $e^{-\phi}dx$ and $\mu$, and in particular that the moment map $\phi$ is basically the solution of a variant to the following variant of the Monge-Ampere equation called also the {\it Kähler-Einstein} equation (as it has numerous application in Kahler geometry, and in particular in the study of differential complex and symplectic structures on toric Fano manifolds, see e.g. Donaldson \cite{DON} for a nice exposure).
\begin{equation}\label{mongA}
e^{-\phi}=\rho(\nabla \phi)det(\nabla^2 \phi)
\end{equation}
where $\rho$ is the (positive) density w.r.t the Lebesgue measure of $\mu$ on its support which is supposed to be a compact, open convex set. 
\newpage
The previous remark stating that the the standard gaussian is the unique fixed point of the previously defined mapping is thus easily deduced from the following remark:
\newline
Indeed, if we fix $\rho$ as the gaussian density, then the unique solution to the Monge-Ampère PDE \eqref{mongA} is given by $\frac{\lVert x\rVert^2_2}{2}$.
\begin{theorem}(Fathi, Theorem 2.3 in \cite{mm})\label{3.3}
Let's suppose that $\mu$ has a density $\rho$ w.r.t the Lebesgue measure which is strictly positive on its support, and that the solution $\phi$ to the PDE \eqref{mongA} is supported on the whole space $\mathbb{R}^d$. Then $Hess\:\phi(\nabla \phi^*)=(Hess\:\phi^*)^{-1}$ is a free Stein kernel for $\mu$ with respect to the standard Gaussian measure.
\end{theorem}
\end{flushleft}
In free context and in the one dimensional case it seems plausible that a similar statement should hold, by following the same heuristic,
\newline 
the map $\mu \mapsto \nu_u$ admits an unique fixed point given by the standard semicircular distribution, and thus heuristically if $u\approx \frac{x^2}{2}$, we should expect that $\mu\approx \mathcal{S}(0,1)$. Indeed, we can also construct free Stein kernel with respect to the semicircular potential, thanks to this {\it free moment map}. It is also interesting to note that it has exactly the same form as in the classical case, up to replacing the classical differential operators by their free counterpart, and in particular to define the noncommutative Hessian. In fact, in the free context the noncommutative Hessian takes the form of the operator $\mathcal{J}\mathcal{D}$. 
Then we will show that in the free context, we have a free Stein kernel of the form $(\mathcal{J}\mathcal{D}u^*)^{-1}$ where $u^*$ denotes the Legendre transform of the free moment map.

\begin{theorem}
Let's suppose that $\mu$ is centered, absolutely continuous with respect to the Lebesgue measure and supported on a compact interval. Denoting $u$, the free moment map of $\mu$ with respect to the free Gibbs measure $\nu_{u}$, i.e $\mu=(u')_{\sharp}\nu_{u}$. If $u$ is $\mathcal{C}^2$ and $u''>0$ on $\mathbb{R}$, then:
\begin{equation}
(x,y)\mapsto (\mathcal{J}\mathcal{D}u^*(x,y))^{-1}:=\frac{x-y}{(u^*)'(x)-(u^*)'(y)}=\frac{x-y}{(u')^{-1}(x)-(u')^{-1}(y)}
\end{equation}
for is a free Stein kernel for $\mu$ with respect to the standard semicircular potential $\nu_{\frac{1}{2}x^2}$.
\end{theorem}
\begin{proof}
For a test function $f$, we have:
\begin{equation}
\int u'(x)f(x)d\nu_{u}(x)=\int\int\frac{f(x)-f(y)}{x-y}d\nu_{u}(x)d\nu_{u}(y)
\end{equation}
Take now $g$ a test function and set $f(x)=g(u'(x))$. Then the previous equation become:
\begin{eqnarray}
\int u'(x)g(u'(x))d\nu_{u}(x)&=&\int\int\frac{g(u'(x))-g(u'(y))}{x-y}d\nu_{u}(x)d\nu_{u}(y)\nonumber\\
&=&\int\int\frac{g(u'(x))-g(u'(y))}{u'(x)-u'(y)}\frac{u'(x)-u'(y)}{x-y}d\nu_{u}(x)d\nu_{u}(y)
\end{eqnarray}
\bigbreak
Now put $\tilde{x}=(u^*)'(x)=(u')^{-1}(x)$ and $\tilde{y}=(u^*)'(y)=(u')^{-1}(y)$, and note that that this change of variables also sends $\mu$ to $\nu_{\mu}$ (and thus also for the product measure $d^{\otimes 2}\mu$ and $d^{\otimes 2}\nu_u$). Doing the same change of variables in the left leg, we then obtain:
\begin{equation}
\int xg(x)d{\mu}=\int\int\frac{x-y}{(u^*)'(x)-(u^*)'(y)}\frac{g(x)-g(y)}{x-y}d\mu(x)d{\mu}(y),
\end{equation}
And thus, we obtain the desired conclusion.
\end{proof}

\begin{flushleft}
Now, we recall that Cébron proved in \cite{C} that the free quadratic Wasserstein distance for compactly supported probability measure (which coincide in dimension one with the usual quadratic Wasserstein distance between compactly supported probability measure) to the standard semicircular distribution is controlled by the free Stein discrepancy.
\end{flushleft}
\begin{theorem}(Cébron, proposition 2.7 in \cite{C})
    Let $\mu$ a compactly supported probability measure, then:
    \begin{equation}
        W_2(\mu,\mathcal{S}(0,1))^2\leq \Sigma^*(\mu,\nu_{\frac{x^2}{2}})^2
    \end{equation}
\end{theorem}
In particular, this allows us to bound the free Stein discrepancy by a regularity estimate on the non commutative Hessian of $u$ as follows:
\begin{corollary}
If $\mu$ satisfies the hypothesis of the previous theorem, then the quadratic Wasserstein distance between $\mu$ and the standard semicircular distribution $\mathcal{S}(0,1)$ is bounded as follows:
\begin{eqnarray}
W_2(\mu,\mathcal{S}(0,1))^2\leq \Sigma^*(\mu|\nu_{\frac{x^2}{2}})^2&\leq &\int\int\bigg[\frac{x-y}{(u^*)'(x)-(u^*)'(y)}-1\bigg]^2d\mu(x)d{\mu}(y)\nonumber\\
&\leq & \int\int\bigg[\frac{u'(x)-u'(y)}{x-y}-1\bigg]^2d\nu_{u}(x)d{\nu}_{u}(y),
\end{eqnarray}
\end{corollary}
This means that if we are able to bound the non commutative Hessian $\mathcal{J}\mathcal{D}{u}$ average against $d\nu_v^{\otimes 2}$, we are able to obtain estimates on transport distances to the semicircular law.
\begin{remark}
In the classical setting, Fathi and Mikulincer \cite{FM} have discovered a very interesting phenomenon from the construction of this "{\it moment Stein kernel}". Indeed, it is possible under some mild assumptions (e.g this Stein kernel should be in the positive cone $S_d^{++}(\mathbb{R})$ and uniformly bounded from below), to construct a diffusion process which admits $\mu$ has its {\it unique invariant measure}, and for which the generator also admits other nice properties (its generator has a spectral gap: the smallest (non-zero) eigenvalue is one ...) provided that $\mu$ satisfies the conditions of theorem \eqref{3.3}. Concretely, this means that if $(B_t)_{t\geq 0}$ is a standard Brownian motion on some filtered probability space, then the diffusion process $(X_t)_{t\geq 0}$:
\begin{equation}
    dX_t=-X_tdt+\sqrt{2Hess\:\phi(\nabla \phi)^*}(X_t)dB_t,
\end{equation}
has $\mu$ as unique invariant measure (the uniqueness is then ensured from a standard criterion on the lower bound on this diffusion coefficient). However as noticed by Fathi and Mikulincer \cite{FM}, the "{\it moment Stein kernel}" is not generally globally Lipschitz, but belong generally to some Sobolev space. Using these moment Stein kernels, a variant of the Crippa and De Lellis \cite{CripaD} techniques to study transport equations, and a proxy-condition coming from the crucial Lusin-Lipschitz property, they obtained a new generalization of the Ledoux, Nourdin, Peccati \cite{LNP} estimates relating transport distances and Stein discrepancies in a non-gaussian setting, that is for invariant measures $\mu$ of diffusions which are well-conditioned log-concave measures.
\end{remark}
\begin{remark}
Now it is important to remark that the above constructed free Stein kernels also enjoys these properties: they are positive, and measurable, from the assumption over the function $u$. Thus one can see $(\mathcal{J}\mathcal{D}u^*(.))^{-1}\in L^2(\mathcal{M})\otimes L^2(\mathcal{M}^{op})\simeq HS(L^2(\mathcal{M}))$ over some tracial $W^*$-probability space $(\mathcal{M},\tau)$, which will be supposed to contain a copy of $L(\mathbb{F}_{\infty})$ and to be filtered:
\newline
This means that a free Brownian motion exists $(S_t)_{t\geq 0}$ on $(\mathcal{M},\tau)$, which should also contain its corresponding filtration $(\mathcal{A}_t)_{t\geq 0}$. This will allows us to define the following free SDE.
\end{remark}
\begin{definition}
Let $(X_t^{X_0})_{t\geq 0}$ be the free stochastic process starting at any (bounded) initial data $X_0\in \mathcal{A}_0$ with law $\mu$ and solution of the following free SDE:
\begin{equation}
dX_t=-X_tdt+\sqrt{2(\mathcal{J}\mathcal{D}u^*(X_t))^{-1}}\sharp dS_t,
    \end{equation}
    \end{definition}
    \begin{flushleft}
This free process is in fact a modification of the the free Ornstein-Uhlenbeck process:
\begin{equation}
    dX_t=-X_tdt+\sqrt{2}dS_t,
\end{equation}
which has $\mathcal{S}(0,1)$ has invariant measure, and for which we recall that we have an exponential fast convergence to the equilibrium for the relative entropy via the free Log-Sobolev inequality, and  even for the free quadratic Wasserstein metric thanks to free Talagrand inequality proved in the breakthrough work of Dabrowski \cite{Dab10} by means of free {\it Otto-Villani estimates} for the free analog of $W_2$ between two close points of the free Ornstein-Uhlenbeck diffusion and the relative free Fisher information along the flow.
\bigbreak
Here, the procedure is to switch the usual conventions of a Langevin diffusion: all the information is carried in the diffusion coefficient, and not in the drift.
\newline
It turns out that this free diffusion has a very interesting property. Indeed, its stationary distribution will be $\mu$. Thus, it gives a large panel of compactly supported measures $\mu$, which are supposed to satisfy the assumptions of the previous theorem, and which are the equilibrium measures of free diffusion processes.
        
    \end{flushleft}
\begin{proposition}
Assume that $(J\mathcal{D}u^*)^{-1}$ is bounded from above and below by some positive constants, then $\mu$ is the unique invariant measure of the free diffusion $(X_t)_{t\geq 0}$. 
\end{proposition}
\begin{proof}
We first compute the infinitesimal generator $L$ of this free SDE. Thanks to the free Ito formula of Biane and Speicher (see also the trace formula of Nikitopolous: theorem 3.5.3, and remark 3.5.4 in \cite{ENiki}) 
, we easily deduce that the generator is given for all $\mathcal{C}^2$ bounded functions by:
\begin{eqnarray}
    (Lf)(x)&=&\Delta_{(J\mathcal{D}u^*)^{-1}(X_0)}f(x)-xf'(x),
\end{eqnarray}
Where we denote $\Delta_{(J\mathcal{D}u^*)^{-1}(X_0)}$ as the correction term coming from the free Ito formula.
\newline
Then we are left to prove that $\mathbb{E}_{\mu}(Lf)=0$ for all $f\in \mathcal{C}^2$ bounded functions (where $\mathbb{E}$ denotes the expectation taken under $\mu$), i.e:
\begin{eqnarray}
    \int (Lf)(x)d\mu(x)=0,
\end{eqnarray}
But,
\begin{eqnarray}
    \int (Lf)(x)d\mu(x)&=&\int\int \frac{f'(x)-f'(y)}{x-y}(\mathcal{J}\mathcal{D}u^*(x,y))^{-1}d\mu(x)d\mu(y)-\int xf'(x)d\mu(x)\nonumber\\
    &=&0.
\end{eqnarray}
which is exactly the Stein equation satisfied by the measure $\mu$ in is weak formulation (we suppose that the test functions are in gradient form, i.e $f\rightarrow f'$).
\end{proof}
\qed
\begin{flushleft}
Unfortunately, as we lack of a definition of the probability measure of a tuple of noncommutative random variables, it is not possible to generalize our argument to the multi-dimensional setting.
\end{flushleft}
We are also able in the same way as Fathi did it in section 5 of \cite{mm}, to construct free Stein kernel easily with respect to other free Gibbs state of reference to compare $\mu$ to a free Gibbs state $\nu_V$ with potential $V$, provided that it exists. This will be possible as soon as the pushforward measure $\mu_V$ of the measure $\mu$ by $V'$ will have barycenter at the origin. It also as expected, exactly the free counterpart of the kernel discovered by Fathi in section 5 of \cite{mm}.
\begin{theorem}
Let's $\nu_{V}$ a free Gibbs measure associated with the potential $V$. Assume that $V$ is uniformly convex, and $\mathcal{C}^2$. Suppose also that $\mu$ has a density w.r.t the Lebesgue measure and that its supported on a compact interval. Denote $\mu_V$ the pushforward of $\mu$ by $V'$ and suppose that $\int xd\mu_V(x)=0$. Then
\begin{equation}
(x,y)\mapsto A_V(V'(x),V'(y)).(\mathcal{J}\mathcal{D}V(x,y))^{-1}:=A_V(V'(x),V'(y))\frac{x-y}{V'(x)-V'(y)}
\end{equation}
is a free Stein kernel for $\mu$ with respect to the potential $V$, where $A_V:\mathbb{R}^2\rightarrow \mathbb{R}$ is a free Stein kernel for $\mu_V$ with respect to the standard semicircular potential $\nu_{\frac{1}{2}x^2}$.
\end{theorem}
\begin{proof}
Let's denote $\mu_{V}$ the pushforward of $\mu$ by $V'$. Then we set $g(x)=f((V^*)'(x))$ for an arbitrary test function $f$, and we have:
\begin{eqnarray}\label{gs}
\int_{\mathbb{R}}V'(x)f(x)d\mu(x)&=&\int_{\mathbb{R}}V'(x)g(V'(x))d\mu(x)\nonumber\\
&=&\int_{\mathbb{R}}xg(x)d\mu_{V}(x)\nonumber
\end{eqnarray}
Now let's suppose that $\mu_{V}$ admits free Stein kernels with respect to the standard semicircular potential $\nu_{\frac{1}{2}x^2}$. This is always ensured by Cébron, Fathi and Mai results in \cite{FCM} provided that
\begin{equation}
\int xd{\mu_V}=\int_{\mathbb{R}}V'(x)d\mu=0,
\end{equation}
\newline
Now denote $A_V:\mathbb{R}^2\rightarrow \mathbb{R}$ such a kernel, then \ref{gs} become:
\begin{eqnarray}
\int V'(x)f(x)d\mu&=& \int\int A_V(x,y)\frac{g(x)-g(y)}{x-y}d\mu_V(x)d\mu_V(y)\nonumber\\
&=&\int\int A_V(V'(x),V'(y))\frac{g(V'(x))-g(V'(y))}{V^{'}(x)-V^{'}(y)}d\mu(x)d\mu(y)\nonumber\\
&=&\int\int A_V(V'(x),V'(y))\frac{f(x)-f(y)}{V^{'}(x)-V^{'}(y)}d\mu(x)d\mu(y)\nonumber\\
&=&\int\int A_V(V'(x),V'(y))\frac{x-y}{V'(x)-V'(y)}\frac{f(x)-f(y)}{x-y}d\mu(x)d\mu(y),\nonumber
\end{eqnarray}
Where we used the fact that $V''(x)>0$ for $x\in supp(\mu)$ and thus that $V'(x)-V'(y)\neq 0$, $\mu$ almost everywhere.
\end{proof}
\qed
\section{Higher-Order Stein kernels}
\section{Higher-order free difference quotient}
\begin{flushleft}
In this section, we will investigate a free analog to the notion of "{\it higher-order}" Stein kernels introduced by Fathi in \cite{F} in the classical case, and which introduce higher-order derivatives of test functions in the integration by parts for the standard Gaussian measure. This will allow us to deduce new free functional inequalities. First of all, we need to extend the derivatives operators introduce in section \eqref{sec1} to an higher-order of derivation.
\end{flushleft}
\begin{definition}
We define the differential (Jacobian) of order $k$ for a tuple of noncommutative polynomial $P=(p_1,\ldots,p_n)\in \mathbb{P}^n$ as: 
\begin{equation}
    \mathcal{J}^k P=\left\{\partial^k_{(i_1,\ldots,i_k)} p_j\right\}_{j,i_1,\ldots,i_k=1}^n \in \mathbb{T}_{(k+1,n)}(\mathbb{C}\langle T\rangle^{\otimes {(k+1)}}),
\end{equation}
where $\partial^k_{(i_1,\ldots,i_k)} =(\partial_{i_1}\otimes id^{\otimes k})\circ (\partial_{i_2}\otimes id^{\otimes k-1})\circ\ldots\circ \partial_{i_k}$.
\newline
Here $\mathbb{T}_{(k+1,n)}(\mathbb{C}\langle T\rangle^{\otimes {(k+1)}})$ the space of $(k+1)$ tensors in $n$-noncummutative variables (that is multi-dimensional arrays) whose entries are valued in $\mathbb{C}\langle T\rangle^{\otimes ({k+1})}$.
\end{definition}
Note that we will see this operator acting on tuples of noncommutative polynomials in $X=(x_1,\ldots,x_n)$ through the canonical evaluation homomorphism on tensors, i.e view $J^k_X$ as:
\begin{equation}
J^k_X: \mathbb{C}\langle x_1,\ldots,x_n\rangle^n \rightarrow L^2\bigg(T_{(k+1,n)}(W^*(X))\bigg),
\end{equation}
and view the iterated free difference quotient associated to $X$ as an unbounded operator, denoted in short $\partial_{i_1,\ldots,i_k}^k$ instead of $\partial_{x_{i_1},\ldots,x_{i_k}}^k$, such that:
$\partial_{i_1,\ldots,i_k}:\mathbb{C}\langle x_1,\ldots,x_n\rangle\rightarrow (L^2(W^*(X)))^{\otimes^{(k+1)}}$. 
\begin{remark}
By the results of Dabrowski, Guionnet, Shlyakhtenko \cite{YGS}, we can extend their action on a larger domain as both of theses operator will be unbounded densely defined and closable, provided that the first and second order conjugate variables are bounded (see the important statement number $3$ of lemma 40.3 in \cite{YGS}).
\end{remark}
\begin{flushleft}
We will equip these tensors with their natural $L^2$ inner product which is the following :
\begin{equation}
\langle A,B\rangle_{2}=\sum_{i_1,\ldots,i_k=1}^n \tau^{\otimes (k+1)}(a_{i_1,\ldots,i_k}.b_{i_1,\ldots,i_k}^*),
\end{equation}
When $k=1$, we will denote as usual $\langle A,B\rangle_{HS}$ instead.
\end{flushleft}
\begin{remark}
When we specify a noncommutative distribution, that is a complex valued exponentially bounded (given a real number $R>0$) linear functional from the $*$-algebra of noncommutative polynomials $\mu: \mathbb{P}\rightarrow \mathbb{C}$, which satisfies the following properties:
\begin{enumerate} 
    \item {\it unital:} $\mu(1)=1$,
    \item {\it tracial:} $\mu(PQ)=\mu(QP)$ for any any $P,Q\in \mathbb{P}$,
    \item {\it  positive:} $\mu(PP^*)\geq 0$ for any $P\in \mathbb{P}$,
    \item {\it exponentially bounded:} for any $j_1,\ldots,j_n\in \left\{1,\ldots,n\right\}$, we have $\lvert \mu(t_{j_1}\ldots t_{j_n})\rvert \leq R^n$.
\end{enumerate}
we will denote in short the associated $L^2$-norm on these tensors by setting: 
\begin{equation}
    \langle A,B\rangle_{\mu,k,2}=\sum_{i_1,\ldots,i_k=1}^n \mu^{\otimes (k+1)}(a_{i_1,\ldots,i_k}.b_{i_1,\ldots,i_k}^*),
\end{equation}
\end{remark}
\subsection{Stein kernels associated with Tchebychev polynomials}
\begin{flushleft}

In the following computations, we will show the important role of Tchebychev polynomials, which are the polynomial eigenfunctions of the free Ornstein-Uhlenbeck  operator denoted $L$, i.e $LU_n=-nU_n$, which forms an orthonormal basis of the $L^2(\mathcal{S}(0,1))$, and how they naturally appears in a generalized integration by parts.
We will focus first our arguments in the univariate case ($n=1)$ which is much easier to deal with.
\begin{equation}
    \int U_n(s)U_m(s)\mathcal{S}(0,1)(ds)=\delta_{n,m},
\end{equation}
We also recall that the free Ornstein-Ulhenbeck operator has the following expression where $f$ is a bounded $\mathcal{C}^2$ function:
\begin{equation}
    (Lf)(x)=2\frac{\partial }{\partial x}\int \frac{f(x)-f(y)}{x-y}\mathcal{S}(0,1)(dy)-xf'(x),
\end{equation}
\newline
The Schwinger-Dyson equation in the stronger form (we don't assume that noncommutative polynomials are in gradient form, see the discussion right after definition (1.8) of Fathi and Nelson \cite{FN}) for this standard semicircular case (where $S$ denotes a standard semicircular random variable) and $P\in \mathbb{C}[S]$ is thus:
\begin{equation}
    \langle S,P(S)\rangle_2=\langle 1\otimes 1,\partial P(S)\rangle_{HS},
\end{equation}
The, let's first remark by taking $P\in \mathbb{C}_{s.a}[S]$, and setting $Q=SP$:
\begin{equation}
    \langle S^2,P(S)\rangle_2=\langle 1\otimes 1, \partial Q(S)\rangle_{HS}
\end{equation}
Now it is straightforward to check (and note that the noncommutative derivative is invariant under the "flip" operation, which is the linear extension of the permutation of algebraic tensors products $flip(x\otimes y)=y\otimes x$ for $x,y \in \mathbb{C}[S]$, since $S, P(S)$ trivially commutes), that $\partial Q=flip(\partial Q)=1\otimes P+S.\partial P$, 
\newline
then we use that:
\begin{equation}
    \langle 1\otimes 1, 1\otimes P(S)\rangle_{HS}=\tau(P(S))=\langle 1,P(S)\rangle_2\nonumber
\end{equation}
to get:
\begin{equation}
    \langle S^2-1,P(S)\rangle_2=\langle S\otimes 1, \partial P(S)\rangle_{HS}
\end{equation}
we then iterate again at {\it third order} this relation, and we denote again $Q=SP$ :
\begin{equation}
     \langle (S^2-1)S,P(S)\rangle_2=\langle S\otimes 1, \partial Q(S)\rangle_{HS}
\end{equation}
Now again $\partial Q=1\otimes P+S.\partial P$,
\end{flushleft}
Now it suffices to remark since $S$ is centered that:
\begin{equation}
    \langle S\otimes 1,1\otimes P(S)\rangle_2=\tau(S)\tau(P(S))=0 \nonumber
\end{equation}
For the second term, we have:
\begin{equation}
    \langle S\otimes 1,S.\partial P(S) \rangle_{HS}=\langle S^2\otimes 1,\partial P(S)\rangle_{HS}\nonumber
\end{equation}
combining these facts, 
\begin{equation}
\langle S^3-S,P(S)\rangle_2=\langle S^2\otimes 1,\partial P(S)\rangle_{HS}\nonumber
    \end{equation}
Now by substraction on both side by:
\begin{equation}
   \langle S,P(S)\rangle_2=\langle 1\otimes 1,\partial P(S)\rangle_{HS}\nonumber
\end{equation}
we finally get:
\begin{equation}
    \langle S^3-2S,P(S)\rangle_2=\langle (S^2-1)\otimes 1, \partial P(S)\rangle_{HS}
\end{equation}
Iterating the process up to any order, we find the following kind of relation:
\begin{proposition}
Let $S$ a self-adjoint noncommutative random variable, then $S$ is a standard semicircular random variable if and if only for some $n\geq 1$,
\begin{equation}
    \langle U_n(S),P(S)\rangle_2=\langle U_{n-1}(S)\otimes 1, \partial P(S)\rangle_{HS},
\end{equation}
where the polynomials $(U_n)_{n\geq 0}$ are the Tchebychev polynomials of second kind determined by the following recursion relation:
\begin{equation}
\left\{ \begin{aligned} 
  U_0&=1 \\
  U_1&=t\\
   tU_n&= U_{n+1}+U_{n-1}
\end{aligned} \right.
\end{equation}
\end{proposition}
From the traciality of the state, that $U_n(S)$ is self-adjoint for all $n\geq 0$ and {\it coassociavity} of the free difference quotient, which is also algebraic $\partial: dom(\partial)\rightarrow dom(\partial)\odot dom(\partial)$, and thus satisfies that:
\begin{equation}
    (\partial\otimes id)\circ \partial=(id \otimes \partial)\circ \partial,
\end{equation}
the relation is also true when "flipped", when $P$ is self-adjoint, i.e for $P\in \mathbb{P}_{s.a}$,
\begin{equation}
    \langle U_n(S),P(S)\rangle_2=\langle 1\otimes U_{n-1}(S), \partial P(S)\rangle_{HS}
\end{equation}
This is exactly the "{\it free version 
}" of the generalized integration by parts characterizing the standard gaussian measure, and involving the Hermite polynomials:
\begin{equation}
    \int H_k(x)f(x)d\gamma(x)=\int H_{k-1}f'(x)d\gamma(x)
\end{equation}
where $f$ is a test function, and the Hermite polynomials are given by $H_k(x)=(-1)^ke^{\frac{x^2}{2}}\frac{d^k}{dx^k}({e^{-\frac{x^2}{2}}})$.
\bigbreak
We can also rewrite the previous relation using higher-order derivatives (see Voiculescu proposition 3.9 in \cite{V}) to get the following, for any $p\in \mathbb{C}[S]$:
\begin{equation}
    \langle U_k(S),p(S)\rangle_2=\langle 1^{\otimes (k+1)} ,\partial^kp(S)\rangle_{2},
\end{equation}
this implies in particular that the higher-order conjugate variables of a semicircular variable are given for all $n\geq 1$ by:
\begin{equation}
    \xi^n(S)=U_n(S)
\end{equation}

Before that, since there is no standard definitions (to our best knowledge) when $n>1$ of multivariate Tchebychev polynomials, we must introduce a definition adapted to our setting, which is the following.
\begin{definition}
We define for $k\geq 1$, and $i_1,\ldots,i_k\in \Iintv{1,n}$, the multivariate Tchebychev polynomial $U_k^{i_1,\ldots,i_k}\in \mathbb{P}$ associated to $i_1,\ldots,i_k$ as the unique noncommutative polynomial such that for a standard semicircular system $S=(S_1,\ldots,S_n)$ (in some tracial $W^*$-probability space $(\mathcal{M},\tau)$), we have:
\begin{equation}
U_k^{i_1,\ldots,i_k}(S_{i_1},\ldots,S_{i_k})=\partial_{(i_1,\ldots,i_k)}^{k,*}\bigg(1^{\otimes (k+1)}\bigg),
\end{equation}
where the iterated free difference quotient are defined with respect to this semicircular system, and we remind that for any choice of $i_1,\ldots, i_k\in \{1,\ldots,n\}$, $\partial_{(i_1,\ldots,i_k)}^k$ is a densely defined unbounded closable operator as it has bounded first and second order conjugate variables.
\newline
We also define the associated $n$-tensor which have these polynomials as coefficients:
\begin{equation}
\bar{U}_k=\left(U_k^{i_1,\ldots,i_k}\right)_{i_1,\ldots,i_k=1}^n,
\end{equation}

\end{definition}
This definition strongly suggests that the first notion of free Stein kernel with respect to the semicircular potential is the following one:
\begin{definition}(Stein kernel associated with Tchebychev polynomials)
We define an free Stein kernel of order $k$ associated with Tchebychev polynomials (in a tracial $W^*$-probability space) and with respect to the standard semicircular potential as an element $\tilde{A}_k$ in $L^2(\mathcal{M}\otimes \mathcal{M}^{op})$, such that the following equality is satisfied for any $P \in \mathbb{P}^n$:
\begin{equation}
\langle \bar{U}_k(X),P(X)\rangle_2=\langle  \tilde{A}_k,[\mathcal{J}P](X)\rangle_{HS},
\end{equation}
\end{definition}
\begin{remark}
Indeed, in dimension $n=1$, if $S$ has the standard semicircular distribution, we would get for all $k\in \mathbb{N}_*$, that $\tilde{\tau}_k=U_{k-1}(S)\otimes 1$ (up to the "{\it flip}" operation), which uniquely determines the semicircular law.
Note also that if $\tilde{A}_1,\ldots,\tilde{A}_k$ exists, then $\tilde{A}_{k+1}$ is recursively determined by these kernels from the recursion of Tchebychev polynomials.
\end{remark}
\subsection{A more interesting definition}
\begin{flushleft}
Let us remark that there is a much more convenient framework to define, study and construct free Stein kernel with respect to the semicircular potential (we also include its homothetic deformations) as proposed by Fathi in \cite{F}. This shifted definitions makes notations and applications much lighter.
\end{flushleft}
\begin{definition}
We define a free Stein kernel of order $k$ with respect to the potential $V$ and denoted $A_k$, as an element of $L^2(\mathbb{T}_{(k+1),n}(\mathcal{M}^{\otimes (k+1)}))$ satisfying the following relation for any $P\in \mathbb{P}^n$:
\begin{equation}
        \langle [\mathcal{D}V](X),P(X)\rangle_{\tau}-(\tau\otimes\tau)\circ Tr(([\mathcal{J}P](X))^*)=\langle A_k,[\mathcal{J}^k P](X)\rangle_{2},
\end{equation}
\end{definition}
In fact, if for some $k\geq 1$, $A_k:=0$, then we are left with the "{\it Schwinger-Dyson}" equation characterizing the free Gibbs state (provide existence and uniqueness of this state). Hence, this notion can in some sense "{\it measure}" how far, a tuple of noncommutative random variables is closed to be a free Gibbs state with respect to some potential. And more interestingly, this definition can be seen as integration-by-parts with respect to some free Gibbs state involving higher-order derivatives of noncommutative polynomials (which are in our context, the space of {\it test-functions}).\newline
Note that the definition we used is coherent with the usual definition free Stein Kernel because we have evidently:
\newline
$A_1=\tilde{A}_1-(1\otimes 1^{op})\otimes I_n$ where $\tilde{A}_1$ is a usual free Stein kernel (according to the definition of Fathi and Nelson \cite{FN}).
\bigbreak
The idea behind this work is motivated by the fact that the classical Schwinger-Dyson equation which is satisfied by the semicircular potential has extension to higher-order, where the higher-order conjugate variable are given by Tchebychev polynomials in these semicircular variables. Note also, that as mentioned by Fathi (see the discussion before the definition 1 in \cite{FN}) it is rather convenient to work with our centered definition compared to the preliminary computations we did before, which seems a little bit harder to manipulate.
\begin{remark}\label{imp}
An important and evident fact to notice is the following, we will use it many times in the sequel. If $A_1,\ldots,A_k$ exists, the for any tuple of noncommutative polynomial $P\in \mathbb{P}^n$:
\begin{equation}
    \langle A_{k-1},[\mathcal{J}^{k-1}P](X)\rangle_{2}=\langle A_{k},[\mathcal{J}^{k}P](X)\rangle_{2},
\end{equation}
\end{remark}

Now we are able to define a quantity which will be crucial for establishing free functional inequalities. This notion will characterize in fact how a tuple of {\it n.c} random variables is close to be a free Gibbs state with respect to these potential $V_{\rho}$.
\begin{definition}
The $k$-order free Stein discrepancy is then defined as follows:
\begin{equation*}
\Sigma_k ^*(X|V)=\inf_{A_k}\lVert A_k\rVert_2,
\end{equation*}
where the infinimum is taken over free Stein kernel $A_k$ of order $k$.
\end{definition}

\section{Existence of higher-order free Stein kernels and moment constraints}
The reader might wonder why we restrict our exposition to semicirculars potentials, as the definition seems well adapted to every free Gibbs state.
The main issues to produce a paper with such a level of generality are listed in the following:
\begin{enumerate}
    \item Firstly, one are still unable to even deduce a general bound for the relative free Fisher information along the free Langevin semigroup in terms of the relative free Stein discrepancy for a general self-adjoint potential (even not too far from the semicircular distribution) , that is $V\in \mathbb{P}$ is bounded from above and below, i.e such that there exists $0<c\leq C$ such that $c.(1\otimes 1)\otimes I_n\leq \mathcal{J}\mathcal{D}V\leq C.(1\otimes 1)\otimes I_n$ holds in $M_n(\mathbb{P}\otimes \mathbb{P}^{op})$.
    \newline
    Indeed, if we denote: (here we only consider a noncommutative polynomial $f$, but there has been a lot of progress in defining various analogs of noncommutative smooth functions due to Guionnet, Dabrowski, Jekel, Shlyakhtenko, Gangbo, Li, Nam\ldots in \cite{YGS,JSG, JS2}).
    \newline
    $P_t(f)(X_0)=\tau(f(X_t)|X_0)$ (the free semigroup) associated with the following free diffusion:
    \begin{equation}\label{37}
        X_t=X_0-\int_0^t[DV](X_s)ds+\sqrt{2}dS_t,
    \end{equation}
    We conjecture, following the impressive work of Cheng, Thalmaier and Fang \cite{Cheng}, which have provided a complete answer generalizing the Ledoux, Nourdin and Peccati estimates for the gaussian measures on Euclidean spaces, in the more general setting of complete connected Riemannian manifold $M$, for measures of the form $d\mu=e^{-V}dVol$ where $Vol$ denotes the Riemannian volume form, associated with a smooth potential $V\in \mathcal{C}^2(M)$ which satisfies $Ric_V=Ric-Hess_V\geq K$ where $K$ is a positive constant (see statements in theorem 3.2, corollary 3.3 and theorem 3.5 in \cite{Cheng}), that if such a bounds exists, it should be in the following form, for all $t>0$:
    \begin{equation}
    \Phi^*(X_t|V)^{\frac{1}{2}}\leq C\frac{e^{-Kt}}{\sqrt{1-e^{-2Kt}}}\Sigma^*(X_0|V),
    \end{equation}
    where $C$ is a positive constant and $K$ should be the smallest positive constant such that for all $i,j=1,\ldots,n$:
    \begin{equation}
    \partial_jD_iV\geq K.1\otimes 1^{op}, 
    \end{equation}
Where this last inequality holds in the sense of operator in $\mathbb{P}\otimes\mathbb{P}^{op}$.

    In fact, even in the simplest one dimensional case, it seems for now out of reach. This lack of result is an obstruction to produce the more general version of the free {\it HSI} inequality. It mainly works in the semicircular case because we have a very simple representation of the semigroup (a Mehler type formula).
    \bigbreak
    In fact, and again, the crucial point is that the Tchebychev polynomial (resp. homogeneous Wigner chaos in the infinite dimensional setting) forms an an orthogonal basis of polynomials eigenvalues of the free Ornstein-Uhlenbeck operator.
    \newline 
    Indeed, when $V=\frac{1}{2}\sum t_i^2$ (standard semicircular case), we have in distribution for \ref{37}, and for any $t\geq 0$, $X_t\simeq e^{-t}X_0+\sqrt{1-e^{-2t}}S$ where $S$ is a $n$-standard semicircular system  supposed to be free with $X_0$.
    \item The second item is related to the original proof of the {\it HSI} inequalities by Ledoux, Nourdin and Peccati \cite{LNP} which have strengthened their result for the gaussian measure, by proving such inequalities with respect to invariant measure of general diffusion processes, which satisfy some additional curvature dimension criterion. Indeed, it is well known that in the classical case that the Bakry and Emery criterion $\Gamma_2\geq k\Gamma$ guarantees that a measure satisfies log-Sobolev. However, it is not sufficient to get {\it HSI}. Indeed it requires at least a criterion of the following type $\Gamma_3\geq K\Gamma_2$ for some $K\geq 0$, which is very hard to check in practice, except for Gaussian, Gamma or Uniform distributions (see, e.g, sections $4.2, 4.3$, and $4.4$ in \cite{LNP}). In free probability, there does not even exist a proof of the logarithmic Sobolev inequality via an appropriated notion of $\Gamma$ operators, which makes difficult to prove it except for semicircular potentials. For further details on this topic, the interesting reader could read the discussion of section 11 in \cite{Di}.
\end{enumerate}
\begin{flushleft}
We will only prove the existence for the potential $V_p$ introduced preciously.
Before showing the existence of such higher-order Stein kernels, we will see that their existence necessary implies moments constraints as Fathi \cite{F} proved it in the classical Gaussian case. In other words, existence of higher-order Stein kernel, let's say of order $k$, necessarily implies that the tuple of noncommutative random variables has its moments of order less than or equal to $k$ that matches the ones of a semicircular system of covariance $C=\rho^{-1} I_n$. 
\end{flushleft}
\section{One dimensional case}
\begin{lemma}\label{lma11}
Suppose that $X$ is a self-adjoint operator in $(\mathcal{M},\tau)$, and admits higher-order free Stein kernel up to order $k\geq1$ with respect to the potential $V_{\rho}$, then we have $\tau(P(X))=\tau(P(S_\rho))$ for any $P\in \mathbb{C}_k[t]$. Moreover, if the moment assumption holds for noncommutative polynomials of order $k+1$, these kernels are centered:
$\tau^{\otimes{(k+1)}}(A_k)=0$.
\end{lemma}
\begin{proof}
Let us first explain, the one dimensional case $n=1$, and by homogeneneity (its just dilate the equations by a factor $\rho^{-1}$), we only consider the case $\rho=1$. We denote here respectively $X,S$ a self-adjoint noncommutative random variable and a standard semicircular random variable in some tracial $W^*$-probability space.
\bigbreak
The case $k=1$ is evident as the existence of a free Stein kernel neceassrily implies by Fathi, Cébron and Mai result in \cite{FCM} that $\tau(X)=0=\tau(S)$.
\newline
Now, if we also suppose that the condition $\tau(X^2)=\tau(S^2)=1$ is satisfied, then from the free Stein identity involving a $1$st order free Stein kernel (our shifted definition), we have:
\begin{eqnarray}
    \tau(X^2)&=&\tau(X.X)\nonumber\\
    &=& \tau(1\otimes 1)+\tau\otimes \tau(A_1.(\partial_X X))\nonumber\\
    &=& 1+\tau\otimes \tau(A_1)
\end{eqnarray}
and $A_1$ is centered, i.e $\tau\otimes \tau(A_1)=0$
\bigbreak
For $k=2$, let's suppose that a Stein kernel of order $2$ denoted $A_2$ exists, we then compute as follows:
\begin{eqnarray}
\tau(X^2)&=&\tau(X.X)\nonumber\\
&=&\tau(1\otimes 1)+\tau^{\otimes 3}(A_2.(\partial^2_X X))\nonumber\\
&=&\tau\otimes\tau(1\otimes 1)\nonumber\\
&=&1,
\end{eqnarray}
Then, if we assume that the third moment matches zero (the third moment of the semicircular distribution), i.e $\tau(X^3)=\tau(S^3)=0$, then from the $2$-th order free Stein identity:
\begin{eqnarray}
\tau(X^3)&=&\tau(X.X^2)\nonumber\\
&=&\tau\otimes\tau(X\otimes 1+1\otimes X)+\tau^{\otimes 3}(A_2.(\partial^2_X X^2))\nonumber\\
&=&\tau^{\otimes 3}(A_2)
\end{eqnarray}
Which trivially gives that:
\begin{equation}
    \tau^{\otimes 3}(A_2)=0,
\end{equation}
Thus the basic idea is that when applied to a monomial of order less than or equal to $1\leq l<k$, the second term vanishes, indeed since $\partial^{k}t^l=0$. Moreover if the moment condition holds true at the order $k+1$, we will see that the kernel is necessarily centered as $\partial^{k}t^k=1^{\otimes (k+1)}$.
\newline
By recursion, let´s suppose that the result holds true for $\tau(X^l)=\tau(S^l), 1\leq l\leq k$, then since we suppose that a free Stein kernel of order $k$ exists, we have from the induction hypothesis:
\begin{eqnarray}
    \tau(X^{k+1})&=&\tau(X.X^k)\nonumber\\
   & =&\tau\otimes \tau\bigg(\sum_{j=1}
^{k} X^{j-1}\otimes X^{k-j}\bigg)+\tau^{\otimes(k+1)}( A_k.(\partial^{k+1}_X X^{k}))\nonumber\\
&=&\tau\otimes \tau\bigg(\sum_{j=1}
^{k} S^{j-1}\otimes S^{k-j}\bigg)\nonumber\\
&=&\tau(S^{k+1}),
\end{eqnarray}
\bigbreak
Again, if we assume that $\tau(X^{k+2})=\tau(S^{k+2})$ (it is here that we use the assumption of moments constraints over polynomials of order $k+2$), then by using a stein kernel of order $k+1$ denoted $A_{k+1}$, we necessarily have :
\begin{eqnarray}
    \tau(X^{k+2})&=&\tau(X.X^{k+1})+\tau^{\otimes (k+2)}(A_{k+1}.(\partial_X^{k+1} X^{k+1}))\nonumber\\
    &=& \tau\otimes \tau\bigg(\sum_{j=1}
^{k+1} X^{j-1}\otimes X^{k+1-j}\bigg)+\tau^{\otimes (k+2)}(A_{k+1})\nonumber\\
&=& \tau\otimes \tau\bigg(\sum_{j=1}
^{k+1} S^{j-1}\otimes S^{k+1-j}\bigg)+\tau^{\otimes (k+2)}(A_{k+1})\nonumber\\
&=& \tau(S^{k+2})+\tau^{\otimes (k+2)}(A_{k+1}),
\end{eqnarray}
\bigbreak
Which easily implies that:
\begin{equation}
\tau^{\otimes(k+2)}(A_{k+1})=0,
\end{equation}
\end{proof}
\qed

\section{Multi-dimensional case}
\begin{flushleft}
We recall here that the joint law of a free $(0,\rho^{-1})$-semicircular n-tuple denoted as $S_{\rho}=(S_{\rho}^1,\ldots,S_{\rho}^n)$ is a free Gibbs state with potential $V_p $.
\end{flushleft}
\begin{lemma}\label{lma1}
Suppose that $X$ a $n$-tuple of self-adjoint operators in $(\mathcal{M},\tau)$ admits higher-order free Stein kernels up to order $k\geq1$, then (since moments of any order always exists), we have $\tau(P(X))=\tau(P(S_\rho))$ for any noncommutative polynomials in $n$ variables of degree $l\leq k$. Moreover, if the relation holds for noncommutative polynomials of order $k+1$, these kernels are centered: for any choice of, $i_1,\ldots,i_{k+1} \in \{1,\ldots,n\}$,
\begin{equation}
\tau^{\otimes{(k+1)}}((A_k)_{i_1,\ldots,i_{k+1}})=0.
\end{equation}
\end{lemma}
\begin{proof}
We proceed again by induction,
\bigbreak
For $k=1$, it reduces to the simple condition $\tau([\mathcal{D}V_\rho](X))=\tau([\mathcal{D}V_\rho](S_\rho))=(0,\ldots,0)$, and since $[\mathcal{D}V_\rho](X)=\rho X$, it's obviously true as pointed out by \cite{FCM} as a necessary and sufficient condition for a Stein kernel to exist is that $ \tau([\mathcal{D}V_{\rho}](X))=(0,\ldots,0)$.
\bigbreak
To also obtain the second claim: the first order free Stein kernel is centered (according to our definition), it is sufficient to check the Schwinger-Dyson equation onto the coordinates $t_1,\ldots,t_n$, which is straightforward to verify.
\bigbreak
Now let's suppose the assumption holds for some $k\geq1$,
then by linearity and the induction assumption we can check on monomials of degree $k$.
\bigbreak
The proof that the existence of higher-order Stein kernel of order $k$ implies that the moments of order $1\leq l\leq k$ matches with the ones of a semicircular tuple of covariance $\rho^{-1}I_n$ is then straightforward (see the proof in the univariate case).
\bigbreak
We now focus on the second part of this lemma, which ensures that a free Stein kernel of order $k$ is necessarily centered, which will be a fundamental fact.
\bigbreak
Let's take $P=(0,\ldots,p,\ldots,0)$ with $p$ of degree $k$ at the index $i_{0}$, for an arbitrary $i_{0}\in \{1,\ldots,n\}$. In fact, by linearity one can assume that $p=\rho^{-1}t_{i_k}t_{i_{k-1}}\ldots t_{i_1}$ with $i_1,\ldots,i_k \in \{1,\ldots,n\}$. We then have:
\begin{equation}
    \partial^k_{(j_k,\ldots,j_1)} p=\rho^{-1}.1^{\otimes (k+1)},
\end{equation} if $(j_k,\ldots,j_1)=(i_k,\ldots,i_1)$ and trivially vanishes in the other cases.
\newline
This gives us:
\begin{equation}
    \langle [\mathcal{D}V_{\rho}](X),P(X)\rangle_{\tau}=\tau(x_{i_{0}}x_{i_1}x_{i_2}\ldots x_{i_k}),
\end{equation}
and by the assumption that a higher-order free Stein kernel of order $k$ does exist, we have:
\begin{align*}
      &\tau(x_{i_{0}}x_{i_1}x_{i_2}\ldots x_{i_k})\nonumber\\
      &=\langle A_k,[\mathcal{J}^kP](X)\rangle_2+(\tau\otimes\tau)\circ Tr(([\mathcal{J}P](X))^*)\nonumber\\
     &=\rho^{-1}\tau^{\otimes(k+1)}(A_{k_{(i_{k},\ldots,i_0)}})+\rho^{-1} (\tau\otimes\tau) \bigg(\sum_{l=1}^k\prod_{j\leq l-1} X_{j_l}\otimes\prod_{j>l}X_{j_l}\bigg)\nonumber\\
      &=\rho^{-1}\tau^{\otimes(k+1)}(A_{k_{(i_{k},\ldots,i_0)}})+\rho^{-1}(\tau\otimes\tau) \bigg(\sum_{l=1}^k \prod_{j\leq l-1} S_{\rho}^{j_l}\otimes \prod_{j>l}S_{\rho}^{j_l}\bigg),
     \end{align*}
Where we have used the induction assumption to match the second term, moreover since for the free Gibbs state $V_{\rho}$ the left hand and the second term of the right hand matches, the induction on monomials is proved.
\bigbreak
Moreover, we see that if the moment condition at the order $k+1$ holds true,
since \begin{equation}
\rho^{-1}(\tau\otimes\tau) \bigg(\sum_{l=1}^k\prod_{j\leq l-1} S_{\rho}^{j_l}\otimes \prod_{j>l}S_{\rho}^{j_l}\bigg)=\tau(S^{i_0}_{\rho}\ldots S^{i_k}_{\rho}),
\end{equation}
it implies for any arbitrary choice of $i_0,\ldots,i_{k}$ that (note that we shifted the indexes from $0$ to $1$ to avoid confusions):
$\tau^{\otimes(k+1)}((A_k)_{i_1,\ldots,i_{k+1}})$=0, for every choice of $i_1,\ldots,i_{k+1} \in \{1,\ldots,n\}$.
\newline
We could also notice that the traciality of the state already implies that if the coefficient of the free Stein kernels $\tau^{\otimes(k+1)}((A_k)_{i_1,\ldots,i_{k+1}})=0$, then it is also the case for any cyclic permutation of $i_0,\ldots,i_{k}$.
\end{proof}
\qed
\begin{flushleft}
Now, as in the classical case, we associate to this noncommutative distribution $\mu$ some natural definitions of {\it noncommutative Sobolev spaces}.
When we specialize $\mu(P):=\tau(P(X))$ as the distribution of some tuple $X=(x_1,\ldots,x_n)$ of self-adjoint variables in some tracial $W^*$-probability space, we can think about it as the Sobolev spaces of order "$k$" associated to the distribution of a $n$-tuple of {\it n.c} random variables. We will voluntarily omit to denote the exponent $2$, since we are only focus with $L^2$-closure.
    
\end{flushleft}
\begin{definition}
For any noncommutative distribution $\mu : \mathbb{P} \rightarrow \mathbb{C}$, we associate its following (variant) Sobolev space of order $k$, denoted $H^k(\mu)$ defined as the separation completion of noncommutative polynomials $P\in \mathbb{P}^n$, with respect to the following seminorm:
\begin{eqnarray}
    \lVert P\rVert_{H^k(\mu)}^2:=\lVert \mathcal{J}^kP\rVert_{\mu,k,2}^2,
\end{eqnarray}
\end{definition}
\begin{flushleft}
The main idea that we will use in the following is the free variant of the Poincaré inequality. This last one was proved by Voiculescu in an unpublished note (see the remark of Dabrowski, lemma 2.2 in \cite{DAB}), and it is always satisfied for any choice of self-adjoint tuple of operators (this is a very different fact from what happens in the classical case). For reader convenience, we restate these inequalities.
\end{flushleft}

\begin{definition}(Free Poincaré inequality, Voiculescu)
 A self-adjoint $n$-tuple $X$ satisfies a free Poincaré inequality with constant $C$ if for any noncommutative polynomial $p \in \mathbb{P}$ we have :
 \begin{equation}
 \lVert p(X)-\tau(p(X)) \rVert_2^2 \leq C \sum_{i=1}^n \lVert \partial_i p(X)\rVert_2^2,
 \end{equation}
 The optimal constant verifying this inequality will be denoted by $C_{opt}$.
\end{definition}
\begin{theorem}(Voiculescu, Fathi and Nelson, prop 3.2 in \cite{FN})
Let $X = (x_1,\ldots ,x_n)$ a $n$-tuple of self-adjoint operators. Then, for any self-adjoint polynomial $p \in \mathbb{P}$, we have:
\begin{equation}
     \lVert p(X)-\tau(p(X)) \rVert_2^2 \leq 2n\lVert X \rVert ^2 \sum_{i=1}^n \lVert \partial_i p(X)\rVert_2^2,
\end{equation}

\end{theorem}
\begin{remark}
In full generality, if $p$ is not self-adjoint, the constant $2$ is replaced by $4$. Moreover, when $X$ is a standard semicircular system, it has been shown by Biane that the best possible constant is given in any dimension by $C_{opt}=1$ (which is the same constant as in the classical gaussian Poincare inequality) .
\end{remark}
We also set, to avoid the notation $\lVert .\rVert_{\mu,k,2}$ (when the context is clear: the law $\mu$ of $n$-tuple of self-adjoint elements $X$ is fixed):
\begin{equation}
    \langle P,Q\rangle_{H^k(\mu)}=\langle [J^kP](X),[J^kQ](X)\rangle_{2},
\end{equation}
We will also impose some additional conditions about the noncommutative distribution $\mu$ (and which are usual assumptions in the theory of Dirichlet forms).
\newline
\textbf{Assumptions:}
\begin{enumerate}

\item The tuple of noncommutative random variables $X$ satisfies a generalized Poincaré inequality up to order $k$. That is for all $l\leq k$, and any choice of $i_1,\ldots,i_l \in \left\{1,\ldots,n\right\}$, $\exists\:C_l$ a constant, depending only on $X$ and on $l$ (and independent of the choice of $i_1,\ldots,i_l$), such that for all $p\in \mathbb{P}$:
\begin{align*}
  \bigg \lVert \partial_{i_1,\ldots,i_l}^lp(X)-\tau^{\otimes (l+1)}(\partial_{i_1,\ldots,i_l}^lp(X)).1^{\otimes (l+1)}\bigg\rVert_2^2\leq C_l(l+1)
\sum_{i_{l+1}=1}^n\bigg\lVert \partial_{i_1,\ldots,i_{l+1}}^{(l+1)}p(X)\bigg\rVert^2_2,
\end{align*}
We also denote in the sequel $C_{\max}=\underset{1\leq l\leq k}{\max C_l}$.
\newline
The sequence of Sobolev spaces $(H^j(\mu))_{j\leq k}$ is thus decreasing as a consequence of these free Poincaré inequalities.
\item The Dirichlet form:
\begin{eqnarray}
\mathbb{P}^n&\rightarrow&\mathbb{C}\nonumber\\
    P&\mapsto&\lVert \mathcal{J}P\rVert_{\mu,HS}^2,
\end{eqnarray}
 is closable on the set of noncommutative polynomials. Said otherwise, when the distribution $\mu$ is canonically associated to some tuple $X=(X_1,\ldots,X_n)$ of self-adjoint operators in $(\mathcal{M},\tau)$, a tracial $W^*$-probability space, this tuple have necessarily full non-microstates free entropy dimension, i.e $\delta^*(X)=n$ (see Theorem 2.10 in \cite{CS}). In fact in the next applications we will assume stronger conditions, such as bounded or even Lipschitz conjugate variables as defined by Dabrowski in \cite{DAB14}).
\end{enumerate}
\begin{remark}
In fact, (the proof could be found right after this paragraph) it is straightforward to check the validity of these generalized Poincare inequalities for semicircular systems (with a constant $C_l=1$, for all $l\in \mathbb{N}$) and it is simply based upon the first order free Poincare inequality and that the Tchebychev polynomials forms an orthonormal basis of the noncommutative polynomials in semicircular elements. 
\bigbreak
It is even true in an infinite dimensional setting where noncommutative polynomials in semicircular elements are replaced by regular functionals (in the free Malliavin calculus sense) of a free Brownian motion, and it is again based upon the first-order free Poincare inequality on the Wigner space (see theorem $8$ in \cite{Di2}) and on the chaotic decomposition.
Thus our assumption cover examples of regular functionals of a (possibly infinite) semicircular process (e.g finite sum of Wigner integrals, which we recall that are norm-closure of polynomials in semicircular elements):
\newline
e.g, we can consider $X=(P_1(S_1,\ldots,S_n),\ldots,P_n(S_1,\ldots,S_n))$ for some $P_1,\ldots, P_n \in \mathbb{P}$, for which the the non-microstates free entropy dimension $\delta^*(X)$ is full.
\end{remark}
We only state for convenience the infinite dimensional version (which is the strongest one) of the generalized Poincare inequality on the Wigner space involving Malliavin operators (in fact, we will only prove it for infinitely smooth functionals which is sufficient for our purpose, as noncommutative polynomials in semicircular variables always belong to the space of {\it test Wigner functionals}, see definition 33 and remark 12 in \cite{Di2}). For the terminology about the "{\it Wigner space}", the associated Sobolev Wigner spaces, and basics definitions about the Malliavin operators, we refer to \cite{Di2} and on the first reference about free Malliavin calculus of Biane and Speicher \cite{BS}.
\begin{proposition}
Let $F\in \mathbb{D}^{\infty,2,\sigma}$, a test Wigner functional, then we have the following generalized Poincare inequality which holds true for every integer $k\geq 0$.
\end{proposition}
\begin{equation}
     \bigg \lVert \nabla^kF-\tau^{\otimes (k+1)}(\nabla^k F).1^{\otimes (k+1)}\bigg\rVert_{2}^2\leq (k+1)
\bigg\lVert \nabla^{(k+1)}F\bigg\rVert^2_{2},
\end{equation}
\begin{proof}
Since $F$ admits a chaotic decomposition in terms of multiple Wigner-Ito integral, we have:
\begin{equation}
    F=\sum_{n=0}^{\infty} I_n(f_n)
\end{equation}
Thus applying iteratively the Malliavin operator (note that the conventions of derivations, that is on which side of the tensor product the derivatives applies, are changed from the definition 26 in \cite{Di2}, nevertheless since the Malliavin derivatives is an (almost-everywhere) coassociative derivation, it changes nothing).
\newline
We also have since $F$ is infinitely smooth, from the free Stroock formula (theorem 11 in \cite{Di2}) that:
\begin{equation}
    f_n=\tau^{(n+1)}(\nabla^n F)
\end{equation}
Now we recall that the notations $I_n^{\otimes (i_1,\ldots,i_k)}:=I_{i_1-1}\otimes\ldots\ldots\otimes I_{n-i_k}$ stands for the tensor product of the multiple Wigner integrals with this determined orders.
\begin{equation}
    \nabla^k F=\sum_{n=k}^{\infty}\sum_{1\leq i_1<\ldots<i_k\leq n} I_n^{\otimes (i_1,\ldots,i_k)}(f_n)\nonumber
\end{equation}
And thus 
\begin{equation}
    \nabla^k F-\tau^{(k+1)}(\nabla^kF).1^{\otimes (k+1)}=\sum_{n=k+1}^{\infty}\sum_{1\leq i_1<\ldots<i_k\leq n} I_n^{\otimes (i_1,\ldots,i_k)}(f_n)\nonumber
\end{equation}
Now by computing the $L^2$-norm, we have from the Wigner-Ito isometry:
\begin{equation}
    \bigg \lVert \nabla^k F-\tau^{(k+1)}(\nabla^kF).1^{\otimes (k+1)}\bigg\rVert_2^2=\sum_{n=k+1}^{\infty}\binom{n}{k}\lVert f_n\rVert_2^2\nonumber
\end{equation}
On the other side, 
\begin{equation}
    \bigg \lVert \nabla^{(k+1)}F\bigg\rVert_2^2=\sum_{n=k+1}^{\infty}\binom{n}{k+1}\lVert f_n\rVert_2^2\nonumber
\end{equation}
which suffices to conclude since $(k+1)\binom{n}{k+1}\geq \binom{n}{k}$, for $n>k$.
\end{proof}
\qed
\begin{remark}
The first assumption of a generalized free Poincare inequality is not an assumption in the classical case: it is automatically satisfied  as soon as the probability measure $\mu\in \mathcal{P}(\mathbb{R}^d)$ satisfies a (standard) Poincare inequality: $\exists\:C_P<\infty$, such that for any $f$, a locally Lipschitz function with $\int fd\mu=0$, we have:
\begin{equation}
   \int f^2d\mu\leq C_P \int \lvert \nabla f\lvert^2d\mu
\end{equation}
\end{remark}

\begin{theorem}
Suppose that the self-adjoint $n$-tuple $X=(x_1,\ldots,x_n)$ satisfies the generalized free Poincare inequalities and has its moments of order less than or equal to $max(k,2)$ which match with those of a free $(0,\rho^{-1})$ semicircular system. Then an higher-order free Stein kernel of order $k$ with respect to the potential $V_{\rho}$ does exist.
\end{theorem}
\begin{proof}
We will prove the assertion by induction:
\begin{enumerate}
    \item the case $k=1$ is exactly the existence of a free Stein kernel, from which the result of Fathi, Cébron and Mai in \cite{FCM} ensures that if $\tau([DV_{\rho}])=(0,\ldots,0)$, then a free Stein kernel with respect to the potential $V_{\rho}$ always exists. 
    \item 
    Now, by induction assumption, suppose that a Stein kernel of order "$k$" denoted $A_k$ does exist. We also recall that we denote in a shorthand the constant valued tensor for any tuple $P\in \mathbb{P}^n$ of noncommutative polynomials:
    \begin{equation}
        \tau^{\otimes{(k+1)}}([\mathcal{J}^kP](X)):=\left(\tau^{\otimes{(k+1)}}(\partial^k_{i_1,\ldots,i_k}P_{i_0}(X)).1^{\otimes (k+1)}\right)_{i_0,i_1,\ldots,i_k=1}^n,
    \end{equation}

\bigbreak
We consider then the following functional, the basic idea being to show that it is a continuous linear form with respect to the norm $\lVert.\rVert_{H^{k+1}(\mu)}$: 
\begin{eqnarray}
\mathbb{P}^n &\rightarrow& \mathbb{C}\nonumber\\
    P=(p_1,\ldots,p_n) &\mapsto& \langle A_k,[\mathcal{J}^kP](X)\rangle_{2},
\end{eqnarray}
We have since $A_k$ is centered by the moments assumptions (see lemma \eqref{lma1}):
\begin{eqnarray}
     \bigg\lvert\langle A_k,\mathcal{J}^kP\rangle_{2}\bigg\rvert&=&\bigg\lvert\langle A_k,[\mathcal{J}^kP](X)-\tau^{\otimes{(k+1)}}([\mathcal{J}^kP](X))\rangle_{2}\bigg\rvert \nonumber\\
     &\leq &\lVert A_k\rVert_2\sqrt{\sum_{i_0,i_1,\ldots,i_k=1}^n\bigg\lVert \partial^k_{i_1,\ldots,i_k}P_{i_0}-\tau^{\otimes{(k+1)}}(\partial^k_{i_1,\ldots,i_k}P_{i_0}).1^{\otimes (k+1)}\bigg\rVert_2^2}\\
    &\leq& \sqrt{(k+1)C_{max}} \lVert A_k\rVert_2\lVert \mathcal{J}^{k+1}P\rVert_{2},
\end{eqnarray}
Where we have used the assumptions of the generalized free Poincaré inequalities
\bigbreak
Therefore the previous map 
\begin{eqnarray}
\mathbb{P}^n&\rightarrow& \mathbb{C}\nonumber\\
    P &\mapsto &\langle A_k,[\mathcal{J}^kP](X)\rangle_{2},
\end{eqnarray}
is a continuous linear form with respect to the norm (on the quotient space) $\lVert .\rVert_{H^{k+1}(\mu)}$. From the Riesz representation theorem, we get the existence of $\g\in H^{k+1}(\mu)$ such that:
\begin{equation}
    \langle A_k,[\mathcal{J}^kP](X)\rangle_2=\langle \g,P\rangle_{H^{k+1}(\mu)},
\end{equation}

Hence $\g$ is a right candidate for the kernel.
Indeed, the linear map from $\mathbb{P}^n/\mathcal{N}^k(\mu)$ (where $\mathcal{N}^{k+1}(\mu):=\left\{P\in \mathbb{P}^n,\: \langle P,P\rangle_{H^{k+1}(\mu)}=0\right\} $ to $L^2\bigg(T_{(k+2,n)}(\mathcal{M})\bigg)$:
\begin{equation}
    P \mapsto [\mathcal{J}^{k+1}P](X),
\end{equation}
is an isometry which extends to $H^{k+1}(\mu)$. Moreover denoting by $[\mathcal{J}^{k+1}\g](X)$ the image of such $\g$ via the isomorphism, we have 
\begin{equation}
    \langle A_k,[\mathcal{J}^kP]\rangle_2=\langle [\mathcal{J}^k\g](X),[\mathcal{J}^kP](X)\rangle_{2},
\end{equation}
for any $P \in \mathbb{P}^n$, which concludes the proof.
\end{enumerate}
\end{proof}
\qed
\begin{remark}
Here the construction of these higher-order free Stein kernel is done via the assumption of generalized free Poincaré inequalities, but an explicit construction of free stein kernel of order "$1$" exists as pointed out by Cébron, Fathi, and Mai in theorem 2.1 in \cite{FCM}. We do not know for now (since we are still unable to construct it), but it seems plausible that an explicit construction in fact does exist. If so, and moreover, if such a construction holds for any arbitrary potentials, we could introduce, define and study a generalized notion of {\it free Stein dimension} as defined formally in the section \ref{op}.
\end{remark}
\subsection{Higher-order Stein discrepancies and functional inequalities}

We fix a finite von Neumann algebra $\mathcal{M}$ equipped with a faithful normal trace $\tau$ and a $n$-tuple $X=(x_1,\ldots,x_n)$ of self-adjoint operators in $\mathcal{M}$.
\newline
Provided the existence of conjugate variables for $X$, we will denote them by $\xi_{1},\ldots,\xi_{n}$.
\newline
Let also suppose that $S = (s_1,\ldots,s_n) \in \mathcal{M}^n$ is a free (0,1)-semicircular $n$-tuple supposed free from $x_1,\ldots,x_n$ (if not taking the reduced free product with the free group factor $L(\mathbb{F}_n)$ and thus consider $\mathcal{M}*L(\mathbb{F}_n)$ will provide the existence of these free semicircular elements in this latter von Neumann algebra).
\newline
Fix $\rho> 0$, we will here only focus our exposure on the potential $V_{\rho}$ introduced before.
\bigbreak
We define the following interpolation (Ornstein-Uhlenbeck type) for each $t \geq 0 $ and each $j=1,\ldots,n$,
$x_j(t)=e^{-t}x_j+\sqrt{1-e^{-2t}}\frac{1}{\sqrt{\rho}}s_j$ collected in a $n$-tuple:
\begin{equation}\label{OU}
X(t)=(x_1(t),\ldots,x_n(t)),
\end{equation}
\begin{flushleft}
Let 
$\E_t  : \mathcal{M}\rightarrow W^*(X(t))$ denote the (unique and trace preserving) conditional expectation onto the von Neumann subalgebra $W^*(X(t))$. We remind here that by propositions 3.7, 3.8 and 3.9 of Voiculescu in \cite{V}, the conjugate variable of order $k$ denoted $\xi^k$ in the one dimensional case $n=1$ is given by :
\end{flushleft}
\begin{equation}
\xi^k(t)=\frac{\rho^{\frac{k}{2}}}{{(1-e^{-2t}})^\frac{k}{2}}\E_t(U_k({s})),
\end{equation}
where $U_k$ is again the $k$-order Tchebychev polynomial of second kind.
\newline
And we also have that  also that 
\begin{equation}
    \xi^k(t)=e^{kt}\E_t(\xi^k),
\end{equation}
In the multivariate case, we have by the same Voiculescu's propositions , that the $k$-order conjugate variables are given by:
\begin{equation}
    \xi^k_{j_1\ldots,j_k}(t)=\frac{\rho^{\frac{k}{2}}}{(1-e^{-2t})^{\frac{k}{2}}}\E_t(U_k^{{j_1}\ldots {j_k}}(S_{j_1},\ldots,S_{j_k})),
\end{equation}
In fact, and has ideas from Malliavin calculus will be used in the sequel, we recall thah it is also easily seen by using analysis on the Wigner space that  $U_k^{{j_1}\ldots {j_k}}(S_{j_1},\ldots,S_{j_k})$ is 
in fact an elementary multiple Wigner integrals with respect to the free Brownian motion (which is also a free iterated Skorohod integral) of order $k$ (see Biane and Speicher \cite{BS} or Kemp et a.l. for a complete exposure \cite{KNPS}). To see that, it suffices to consider an orthonormal basis of $L^2_{\mathbb{R}}(\mathbb{R}_+)$, and set for each $j_1\ldots,j_k=1\ldots,n$, the function $f_{j_1,\ldots,j_k}=e_{j_1}\otimes \ldots \otimes e_{j_k}$, which 
satisfies 
$I_k(f_{j_1,\ldots,j_k})=U_k^{{j_1}\ldots {j_k}}(S_{j_1},\ldots,S_{j_k})$.
(see e.g \cite{Di2} for furthers statement about free
Malliavin calculus on the finite dimensional semicircular spaces).
\bigbreak
Now, we are able to generalize some functional inequalities. The first will be the free {\it HSI} inequality proved by Fathi, Neslon (2016) and refined by the author in \cite{Di} to semicircular families with arbitrary covariance matrices. We restate for convenience their theorem and their main lemmas.
\begin{theorem}(Fathi, Nelson theorem 2.6 in \cite{FN})
Let $X=(x_1,\ldots,x_n)\in \mathcal{M}_{s.a}^n$ and $S=(s_1,\ldots,s_n)$ a free (0,1) semicircular system, then for any $\rho > 0 $ we have :
\begin{equation}
    \chi^*(X|V_p)-\chi^*(\sqrt{\rho^{-1}}S|V_p) \leq \frac{1}{2}\Sigma^*(X|V_p)^2 \log\left(1+\frac{\Phi^*(X|V_p)}{\rho\Sigma^*(X|V_p)^2}\right),
\end{equation}
\end{theorem}
The proof of this theorem is simply achieved by the following powerful lemma which gives a bound of the free Fisher information relative to these potentials along the Ornstein-Uhlenbeck flow:
\begin{lemma}(Fathi, Nelson lemma 2.4 in \cite{FN})\label{lem1}
Suppose $\Phi^*(X|V_p)<\infty$. For $\rho>0$ and $X(t)=(x_1(t),\ldots,x_n(t))$ defined previously (\ref{OU}), we have:
\begin{equation}
    \frac{1}{p}\Phi^*(X(t)|V_p)\leq \frac{e^{-4t}}{1-e^{-2t}}\Sigma^*(X|V_p)^2,
\end{equation}

\end{lemma}

We can now obtain a new upper bound for the relative free Fisher information along the associated Ornstein-Uhlenbeck semigroup in terms of $k$-th free Stein discrepancy:
\begin{lemma}\label{lma4}
Let suppose that $X=(x_1,\ldots,x_n)\in \mathcal{M}_{s.a}^n$ has first and second order bounded conjugate variables, then for any $t>0$, we have :
\begin{equation}
\frac{1}{\rho^k}\Phi^*(X(t)|V_p)\leq  \frac{e^{-2(k+1)t}}{(1-e^{-2t})^k}\Sigma_k ^*(X|V_p)^2,
\end{equation}
\end{lemma}
\begin{proof}
We first sketch the arguments which give a representation of the relative free Fisher information along the Ornstein-Uhlenbeck flow in terms of higher-order Stein kernel and which are valid in every dimension $n\geq 1$. We however only specify (and only give details) the case $n=1$ for sake of clarity because even in the one dimensional case the proof already involves a lot of notations.
\bigbreak
We can also assume finite first order free Stein discrepancy $\Sigma^*(X|V_{\rho})$ as if not the case there is nothing to prove. 
\bigbreak
Then in the proof of the previous lemma of Fathi and Nelson (see lemma 2.4 in\cite{FN}), we have the following representation of the relative free Fisher information along the Ornstein-Uhlenbeck flow:
\begin {equation}
\Phi^*(X(t)|V_p)=e^{-t}\langle \mathcal{J}_X^*[(1\otimes1)\otimes I_n-A],\Xi(t)-\rho X(t)\rangle_2
\end{equation}
where $A$ is a free Stein kernel (of order $1$) and $\Xi=(\xi_1,\ldots,\xi_n)$ with each $\xi_i \in L^2(W^*(X))$ are the conjugate variables of $X$ collected in a tuple. So $\mathcal{J}^*_X((1\otimes 1)\otimes I_n)=\Xi$.
\begin{flushleft}
Now using the higher-order Stein kernel $A_k$, and denoting $ \mathcal{J}_X^{k,*}$ the adjoint of $\mathcal{J}_X^k$, and by the remark \ref{imp}, $\langle A_k,[\mathcal{J}^k P](X)\rangle_{2}=\langle A-(1\otimes 1)\otimes I_n,[\mathcal{J}P]X)\rangle_{HS}$ we have :
\begin{equation}
    \mathcal{J}_X^*[(1\otimes1)\otimes I_n-A]=\mathcal{J}_X^{k,*}(-A_k),
\end{equation}
\end{flushleft}
\begin{flushleft}
Then, it is readily checked on $\mathbb{C}\langle x_1(t),\ldots,x_n(t)\rangle$ that,
\begin{equation}
    e^{kt}\mathcal{J}_X^k=\mathcal{J}_{X(t)}^k=\frac{\rho^{\frac{k}{2}}}{{(1-e^{-2t}})^{\frac{k}{2}}}\mathcal{J}_S^k,
\end{equation}
\end{flushleft}
And so, for any $Z\in \huge{\mathbb{E}}_t(\mathcal{M})^n$
\begin{equation}
    \langle \mathcal{J}_X^*[1-A],Z\rangle_2=\langle \mathcal{J}_X^{k,*}(-A_k),Z\rangle_2= \frac{e^{-kt}\rho^{\frac{k}{2}}}{{(1-e^{-2t}})^{\frac{k}{2}}}\langle \mathcal{J}_{S}^{k,*}(-A_k),Z\rangle_2,\nonumber\\
\end{equation}
Now, we have to remind the following property: $\partial_{s_j}|_{W^*(X)}\equiv 0$ for all $j=1,\ldots,n$.
Then, :
\begin{equation}
    \mathcal{J}_S^{k,*}(A_k)=\left\{\sum_{i_1,\ldots,i_k=1}^n\partial^{k,*}_{(s_{i_1},\ldots,s_{i_k})} (A_k)_{i,i_1,\ldots,i_k}\right\}_{i=1}^n,
\end{equation}
Note also that we can compute recursively the $k$th order adjoint:
\begin{equation}
    \partial^{k^*}_{(s_{i_1},\ldots,s_{i_n})}=\partial_{s_{i_1}}^*\circ ( id\otimes \partial_{s_{i_{2}}}^*)\circ\ldots\circ (id^{\otimes ({k-1})}\otimes \partial_{s_{i_k}}^*),\nonumber\\
\end{equation}
Thus to reach the conclusion, it suffice to prove that:
\begin{equation}
    \lVert \mathcal{J}_S^{k,*}(A_k)\rVert_2=\lVert A_k\rVert_2,
\end{equation}
For sake of clarity, in the sequel we focus in the one dimensional case (at the cost of much heavier notations, the proof also holds in the multi-dimensional setting).
\begin{flushleft}
\textbf{Case $n=1$}
\end{flushleft}
\begin{flushleft}
The next ideas to pursue the analysis are borrowed from what happens in the free Malliavin calculus, and in particular they are deeply connected to some explicit formulas about the action of the Mallavin derivative onto {\it homogeneous Wigner chaos}: Ito-isometry, Heisenberg commutation relation, Wick decomposition... (see for example \cite{Di2} for a complete exposure), which in this case consider that the spaces of elementary integrands are given by $\mathbb{C}.1^{\otimes (k+1)}, k\geq 0$.
\bigbreak
The case $k=1$ is immediate, and already fully understood by Fathi and Nelson work. It is mostly based on the following fact:
If $a,b \in \mathbb{C}[X]$, with $X,S$ free and $S$ a semicircular variable in some $W^*$-probability space, then the map:
\begin{equation}
    a\otimes b\mapsto aS b,
\end{equation}
which extends isometrically onto $L^2\bigg((W^*(X))^{\otimes 2}\bigg)$
\bigbreak
Let's first explain the case $k=2$ and $n=1$ , then, remark that for $a,b,c \in \mathbb{C}[X]$ and $S$ a semicircular variable in some $W^*$-probability space supposed to be free with $X$, it is straightforward to check that:
\begin{equation}
\partial _S^{2,*}(a\otimes b\otimes c)=aSbSc-\tau(b)ac
\end{equation}
which is centered $\tau(\partial _S^{2,*}(a\otimes b\otimes c))=\tau(aSbSc-\tau(b)ac)=0$ and thus that:
\begin{equation}\label{713}
\lVert aSbSc-\tau(c)ab\rVert_2^2=\lVert a\otimes b\otimes c\rVert_{L^2(\tau^{\otimes 3})},
\end{equation}
Thus this isometry \eqref{713} extends to the whole space $L^2\bigg((W^*(X))^{\otimes 3}\bigg)$.
\bigbreak
This explicit expression of $\partial _S^{2,*}(a\otimes b\otimes c)$ and a simple computation shows that this second-order divergence is clearly orthogonal in the $L^2$-sense to the operator valued chaos of order $0$: $\mathcal{H}_0=L^2(W^*(X))$, and of order $1$:
$\mathcal{H}_1=\left\{\partial_S^{1,*}(U)=U\sharp S, /U\in L^2((W^*(X))^{\otimes 2})\right\} $, 
\newline
This also implies more generally that:
i.e for $a,b,u,v,z \in \mathbb{C}[X]$,
\newline$\tau[(\partial _S^{2,*}(a\otimes b))z]=0$ and
$\tau[(\partial _S^{2,*}(a\otimes b)).(uSv)]=0$.
\newline
Remark also that a simple application of the chain rule leads to the following equality on algebraic tensor products of elements in $\mathbb{C}[X]$:
\begin{equation}
    \partial_S^2\partial_S^{2,*}(a\otimes b\otimes c)=a\otimes b\otimes c
\end{equation}
which allows to deduce in an easier way the isometry \eqref{713}. Indeed, we have:
\begin{eqnarray}
    \lVert\partial_S^{2,*}(a\otimes b\otimes c)\lVert_2^2&=&\langle \partial_S^{2,*}(a\otimes b\otimes c),\partial_S^{2,*}(a\otimes b\otimes c)\rangle_2\nonumber\\
    &=& \langle \partial_S^2\partial_S^{2,*}(a\otimes b\otimes c),a\otimes b\otimes c\rangle_{L^2(\tau^{\otimes 3})}\nonumber\\
    &=& \lVert a\otimes b\otimes c\rVert_{L^2(\tau^{\otimes 3})}^2,
\end{eqnarray}
\end{flushleft}
Now the idea is to use prove by induction the following proposition (as we remarked previously, it suffices to prove ($4$), which will be itself seen to be a straightforward consequence of ($3$). However, as the other statements might be of independent interest, we choose to state the four following properties satisfies by these operator-valued divergence):
\bigbreak
$\mathcal{P}(k)$: for $k\geq 2$ and $a_1,\ldots,a_{k+1} \in \mathbb{C}[X]$:
\begin{enumerate}
\item $\tau(\partial _S^{k,*}(a_1\otimes \ldots\otimes a_{k+1}))=0$,
\item Orthogonality property: $\partial _S^{k,*}(a_1\otimes \ldots\otimes a_{k+1})$ is orthogonal in the $L^2$-sense to the {\it operator-valued chaos}
$\mathcal{H}_l=\left\{\partial _S^{l,*}f_l, f_l \in L^2((W^*(X))^{\otimes (l+1)})\right\}$, for all $0\leq l<k$.
\item The iterated divergence satisfies the following {\it Heisenberg commutation relation}. For $f=a_1\otimes \ldots\otimes a_{k+1}\in (\mathbb{C}[X])^{\odot (n+1)}$, we have:
\begin{equation}
\partial_S\partial_S^{k,*}f=\sum_{i=1}^k (\partial_S^{i-1,*}\otimes \partial_S^{k-i,*})(f),
\end{equation}
where by convention $\partial_S^{0,*}=id_{\mathbb{C}[X]}$,
\item We have the following {\it Ito-isometry}:
\begin{equation}
\lVert\partial _S^{k,*}(a_1\otimes \ldots\otimes a_{k+1})\rVert_2=\lVert a_1\otimes \ldots\otimes a_{k+1}\rVert_2,
\end{equation}
\end{enumerate}
 Note that the {\it Heisenberg} commutation $(3)$ relation is in fact, a generalized version of the action of the free difference quotient onto Tchebychev polynomials, or equivalently of the action of the free Malliavin derivative of a Wigner integral of order $k$, see example \cite{BS} prop 5.3.9):
 \begin{equation}
\partial_S U_k(S)=\sum_{i=1}^k U_{i-1}(S)\otimes U_{k-i}(S),
 \end{equation}
More importantly, for parallelism with in the free Malliavin Calculus, $\partial_S^{k,*}f$ correspond to an operator-valued divergence (which is in the deterministic case represented a multiple Wigner integral of order $k$). 
\bigbreak
Note that we understand, for each $i=1,\ldots,k$, $\partial_S^{i-1,*}$ as acting on the first "$i-1$ "  of the tensors $f$ and the same for $\partial_S^{k-i,*}$. Note also that the following computation correspond to a Wick-decomposition of an elementary multiple integral of order $k$ represented here by $\partial_{S}^{k,*}(a_1\otimes \ldots\otimes a_{k+1})$ in terms of product of integrals of lower-orders.
\bigbreak
Now, let's $a_1,\ldots,a_{k+1} \in \mathbb{C}[X]$, then (we remind here that $k\geq 2$):
\begin{eqnarray*}
\partial _S^{k,*}(a_1\otimes \ldots\otimes a_{k+1})
&=&\partial_{S}^*\bigg( (id\otimes \partial_{S}^{k-1,*})(a_1\otimes \ldots\otimes a_{k+1})\bigg)\nonumber\\
&=&\partial_{S}^*\bigg( a_1\otimes  \partial_{S}^{k-1,*}(a_2\otimes\ldots\otimes a_{k+1})\bigg)\nonumber\\
&=&a_1S\partial_{S}^{k-1,*}(a_2\otimes\ldots\otimes a_{k+1})-m_1\circ (id\otimes \tau\otimes id)\bigg(a_1\otimes \partial_S\partial_{S}^{k-1,*}(a_2\ldots\otimes a_{k+1})\bigg)\nonumber\\
&=&a_1S\partial_{S}^{k-1,*}(a_2\otimes\ldots\otimes a_{k+1})-a_1.(\tau\otimes id)\bigg(\partial_S\partial_{S}^{k-1,*}(a_2\otimes\ldots\otimes a_{k+1})\bigg)\nonumber\\
&=&a_1S\partial_{S}^{k-1,*}(a_2\otimes\ldots\otimes a_{k+1})-\tau(a_2)a_1.\partial_{S}^{k-2,*}(a_3\otimes\ldots\otimes a_{k+1}),\nonumber
\end{eqnarray*}
Where we used in the fifth line the fact that:
\newline
$(\tau\otimes id)\bigg(\sum_{i=1}^k{\partial_S^{i-1,*}}\otimes {\partial_S^{k-i}}\bigg)(\mathbb{C}[X])^{\odot (k+1)}$ is non-zero if and if only $i=1$, from the fact that for $1\leq l\leq k-1$, $\tau(\mathcal{H}_l)=0$ (the operator-valued chaos of order greater than or equal to one are centered).
\bigbreak
\begin{enumerate}
    \item The iterated divergence of order $k$ is then trivially centered. This follows easily from the fact that we supposed that $k\geq 2$, and the isometry property, since $a_1S=((a_1\otimes 1)\sharp S)\overset{L^2}{\perp} \partial_{S}^{k-1,*}(a_2\otimes\ldots\otimes a_{k+1})$ and
$a_1\overset{L^2}{\perp}\partial_{S}^{k-2,*}(a_2\otimes\ldots\otimes a_{k+1})$ by induction hypothesis.
\item Now let's prove that the {\it Heinseberg commutation relation} is still true at the order $k$. Then, by the induction hypothesis, the previous recursion, and the chain rule satisfied by the free difference quotient. Indeed, we have:
\begin{align*}
&\partial_S\partial_S^{k,*}(a_1\otimes \ldots\otimes a_{k+1})
\nonumber\\
&=\partial_S\bigg[a_1S\partial_{S}^{k-1,*}(a_2\otimes\ldots\otimes a_{k+1})-\tau(a_2)a_1.\partial_{S}^{k-2,*}(a_3\otimes\ldots\otimes a_{k+1})\bigg]\nonumber\\
&=a_1\otimes \partial_{S}^{k-1,*}(a_2\otimes\ldots\otimes a_{k+1})+a_1S.\partial_S\partial_{S}^{k-1,*}(a_2\otimes\ldots\otimes a_{k+1})\nonumber\\
&-\tau(a_2)a_1.\partial_S\partial_{S}^{k-2,*}(a_3\otimes\ldots\otimes a_{k+1})\nonumber\\
&=(\partial_{S}^{0,*}\otimes \partial_{S}^{k-1,*})(a_1\otimes a_2\otimes\ldots\otimes a_{k+1})
+a_1S.\sum_{i=1}^{k-1} (\partial_S^{i-1,*}\otimes \partial_S^{k-i-1,*})(\hat{f}_1)\nonumber\\
&-\tau(a_2)a_1.\sum_{i=1}^{k-2} (\partial_S^{i-1,*}\otimes \partial_S^{k-i-1,*})(\hat{f}_{1,2})\nonumber\\
&=(\partial_{S}^{0,*}\otimes \partial_{S}^{k-1,*})(f)
+\sum_{i=1}^{k-1} (\partial_S^{i,*}\otimes \partial_S^{k-i-1,*})(f)+\tau(a_2)a_1.\sum_{i=2}^{k-1} (\partial_S^{i-2,*}\otimes \partial_S^{k-i-2,*})(\hat{f}_{1,2})\nonumber\\
&-\tau(a_2)a_1.\sum_{i=1}^{k-2} (\partial_S^{i-1,*}\otimes \partial_S^{k-i-2,*})(\hat{f}_{1,2})\nonumber\\
&=(\partial_{S}^{0,*}\otimes \partial_{S}^{k-1,*})f
+\sum_{i=1}^{k-1} (\partial_S^{i,*}\otimes \partial_S^{k-i-1,*})(f)+\tau(a_2)a_1.\sum_{i=1}^{k-2} (\partial_S^{i-1,*}\otimes \partial_S^{k-i-2,*})(\hat{f}_{1,2})\nonumber\\
&-\tau(a_2)a_1.\sum_{i=1}^{k-2} (\partial_S^{i-1,*}\otimes \partial_S^{k-i-2,*})(\hat{f}_{1,2})\nonumber\\
&=(\partial_{S}^{0,*}\otimes \partial_{S}^{k-1,*})(f)+\sum_{i=1}^{k-1} (\partial_S^{i-1,*}\otimes \partial_S^{k-i,*})(f)\nonumber\\
&=\sum_{i=1}^{k} (\partial_S^{i-1,*}\otimes \partial_S^{k-i,*})(f),\nonumber
\end{align*}
Where we have denoted $f=a_1\otimes a_2\otimes\ldots\otimes a_{k+1}$, and its associated tensors:
\newline
$\hat{f}_1:=a_2\otimes\ldots\otimes a_{k+1}$, and $\hat{f}_{1,2}:=a_3\otimes\ldots\otimes a_{k+1}$, and where we used in in the second line the previous Wick-decomposition extended to the bi-tensor case, which in simple words corresponds (for a reader familiar wiith Malliavin calculus) to the product formula for bi-multiple Wigner integrals as proven in the standard case by Bourguin and Campese (see theorem 3.5 in \cite{BC}).
\item
Now by using the following property, which is just a trivial and straightforward iteration of the {\it Heisenberg commutation relation} ($3$):
\newline
We have for all $k\in \mathbb{N}$, and $f\in (\mathbb{C}[X])^{\odot (k+1)}$:
\begin{eqnarray}
    \partial_S^k\partial_S^{k,*}f=f,
\end{eqnarray}
Now, it is evident by using the definition of the adjoint that the isometry is verified since:
\begin{eqnarray}
\lVert \partial _S^{k,*}(a_1\otimes \ldots\otimes a_{k+1})\rVert_2^2&=&\langle \partial _S^{k,*}(a_1\otimes \ldots\otimes a_{k+1}),\partial _S^{k,*}(a_1\otimes \ldots\otimes a_{k+1})\rangle_2\nonumber\\
&=& \langle \partial _S^{k}\partial _S^{k,*}(a_1\otimes \ldots\otimes a_{k+1}),a_1\otimes \ldots\otimes a_{k+1}\rangle_{L^2(\tau^{\otimes (k+1)})}\nonumber\\
&=&
\lVert a_1\otimes \ldots\otimes a_{k+1}\rVert_2^2,
\end{eqnarray}
and thus this equality extends isometrically on $L^2\bigg(W^*(X)^{\otimes (k+1)}\bigg)$.
\newline
Note that this relation also provides the orthogonality of the {\it chaos}. Indeed for $k,l\geq 1$ (for simplicity we can assume $k>l$), $a_1,\ldots,a_{k+1} \in \mathbb{C}[X]$, and $b_1,\ldots,b_l \in \mathbb{C}[X]$, we have:
\begin{align}
   &\langle \partial_S^{k,*}(a_1\otimes \ldots\otimes a_{k+1}),\partial_S^{l,*}(b_1\otimes\ldots \otimes b_{l+1})\rangle_2\nonumber\\
   &=\langle \langle \partial_S^l\partial_S^{k,*}(a_1\otimes \ldots\otimes a_{k+1}),b_1\otimes\ldots \otimes b_{l+1}\rangle_2\nonumber\\
   &=0,
\end{align}
from the orthogonality of the chaos at the order $1$ (and extended to the tensor case).
\end{enumerate}
Now, the conclusion follows by a Cauchy-Scharwz inequality, since we have:
\begin{eqnarray}
\Phi^*(X(t)|V_{\rho})&=&\frac{e^{-(k+1)t}\rho^{\frac{k}{2}}}{{(1-e^{-2t}})^{\frac{k}{2}}}\langle \partial^{k,*}_S(-A_k),\Xi(t)-\rho X(t)\rangle_2\nonumber\\
&\leq &\frac{e^{-(k+1)t}\rho^{\frac{k}{2}}}{{(1-e^{-2t}})^{\frac{k}{2}}} \lVert \partial^{k,*}_S(A_k)\rVert_2 \Phi(X(t)|V_{\rho})^{\frac{1}{2}}\nonumber\\
&\leq &\frac{e^{-(k+1)t}\rho^{\frac{k}{2}}}{{(1-e^{-2t}})^{\frac{k}{2}}} \lVert A_k\rVert_{2}\Phi^*(X(t)|V_{\rho})^{\frac{1}{2}},
\end{eqnarray}
Then, by minimizing over all possible free Stein kernels of order $k$, we finally reach the conclusion.
\end{proof}
\qed

\begin{remark}
The reader familiar with free Stochastic analysis onto semicircular spaces can also remark that from the previous computations we have in fact recovered the recursion for Tchebychev polynomials (which is in fact, a Wick decomposition for the standard iterated divergence). Indeed, taking $a_1=\ldots=a_{k+1}=1$, (recall also that $\mathbb{C}.1$ is trivially free from $W^*(S)\simeq L(\mathbb{Z})$, this important fact is at the core of our more general setting), and denoting 
\newline
$U_k(S)=\partial_S^{k,*}(1^{\otimes (k+1)})$, we have:
$\partial^{k,*}(1^{\otimes (k+1)})=S\partial_S^{k-1,*}(1^{\otimes k})-\partial_S^{k-2,*}(1^{\otimes (k-1)})$.
\newline
All the previous computations could be understand as a kind of {\it free-operator-valued} noncommutative differential calculus onto a semicircular space, where the terms "{\it free}" refers to the freeness between the von Neumann algebra generating by the integrands and the one generated by the integrator.
\newline
In the free Malliavin calculus sense, the spaces of elementary integrands of the iterated divergence (Skorohod integral) are the deterministic {\it multiprocesses} (c.f remark 5.19 in \cite{Di2}), and are given in the usual finite-dimensional semicircular case by $\mathbb{C}.1^{\otimes (k+1)}, k\geq 0$. In our new setting, these are replaced by there operator-valued coefficient $(\mathbb{C}[X])^{\odot (k+1)}, k\geq 0$, and all works approximately well thanks to {\it freeness} between $W^*(X)$ and $W^*(S)\simeq L(\mathbb{Z})$.
\end{remark}
\begin{flushleft}
Note that as expected in the context of free probability where a lot of combinatorial formulas are simplified, we derived a bound without a constant depending on the order of the kernel, in fact we have dropped the factor $\textbf{k!}$, which does appear in the classical case. This factor being a consequence of the integration-by-parts, and that the $L^2$-norms of multivariate Hermite polynomials is exactly $\sqrt{k!}$, (see Fathi lemma 3.2 in \cite{F}). In the free case the $L^2$-norms of Tchebychev polynomials is always one, thus giving an intuition about the absence of such a factor $\textbf{k!}$ (recall that in the free context we are only dealing with the lattice of {\it non-crossing partitions}).
\end{flushleft}
\begin{flushleft}
We now state new variants of the {\it HSI} inequalities. The main tool is the {\it De Brujin's formula} which relates the relative non-microstates free entropy as the integrated free Fisher information along the Ornstein Uhlenbeck flow.
\end{flushleft}
\begin{theorem}(Generalized HSI inequalities)
Let $X=(x_1,\ldots,x_n)\in \mathcal{M}^n_{s.a}$ a tuple of self-adjoint operators which has bounded first and second-order variables and admitting $k$-order free Stein kernel, and $S=(s_1,\ldots,s_n)$ a free $(0,1)$-semicircular system. Then,
we have the following generalized "{\it HSI}" inequalities:
\begin{equation}
 \chi^*(X|V_1)-\chi^*(S|V_1) \leq \frac{1}{2}\min(\Phi^*(X|V_1),\Phi^*(X|V_1)^{\frac{k-1}{k}}\Sigma_k^*(X|V_1)^{\frac{2}{k}})
 \end{equation}
\end{theorem}
\begin{proof}
Indeed, we begin with the following {\it De Brujin's formula} proved by Fathi and Nelson in lemma 2.1 in \cite{FN}, and by the lemma \ref{lma4}, we deduce that:
\begin{align*}
&\chi^*(X|V_1)-\chi^*(S|V_1)\nonumber\\
&=\int_0^{\infty}\Phi^*(X(t)|V_1)dt\nonumber\\
& \leq \Phi^*(X|V_1)\frac{(1-e^{-2t})}{2}+\Sigma_k^*(X|V_1)^2\int_t^{\infty}\frac{e^{-2(k+1)s}}{(1-e^{-2s})^k}ds\nonumber\\
& \leq\Phi^*(X|V_1)\frac{(1-e^{-2t})}{2}+\Sigma_k^*(X|V_1)^2\frac{e^{-2kt}}{(1-e^{-2s})^{k-1}}\nonumber\\
& \leq \Phi^*(X|V_1)\frac{(1-e^{-2t})}{2} + \frac{\Sigma_k^*(X|V_1)^2}{2(1-e^{-2t})^{k-1}}\nonumber
\end{align*}
Then, by optimizing in "$t$", such that 
\begin{equation}
1-e^{-2t}=\left(\Sigma_k^*(X|V_1)^2/\Phi^*(X|V_1)\right)^{\frac{1}{k}}
\end{equation}
if it´s possible. If not, we are left with "$t=+\infty$" and the inequality is reduced to the {\it free logarithmic Sobolev inequality}.
\newline
By replacing the obtained value of "$t$" in this inequality, the conclusion follows.
\end{proof}
\qed
\begin{flushleft}
We can also obtain new transport inequalities, usually called {\it WS : Wasserstein-Stein discrepancies} inequalities (these inequalities first appeared in the classical case for the distance $W_2$ in the breakthrough work of Ledoux, Nourdin and Peccati in \cite{LNP}, for $W_1$ the result was already well-known and at the basis of the Stein's method), which involve both these free Stein discrepancies and the free analogue of the quadratic Wasserstein distance between a tuple of self-adjoint noncommutative random variables and a free semicircular system, and which was first introduced by Biane and Voiculescu \cite{BV}, studied by Dabrowski (in the proof of the free Talagrand inequality \cite{Dab10}), then by Gangbo, Jekel, Nam and Shlyakhtenko (which have proved a free analog of the Monge-Kantorovitch duality), Cébron (WSH inequality \cite{C}) or recently by the author in \cite{Di} (to prove quantitative {\it Fourth moment theorems} on the Wigner space).
\end{flushleft}

Before stating the result, we first recall the definition of the free quadratic Wasserstein distance introduced by Biane and Voicuelscu in \cite{BV}, and which is defined via couplings.
\begin{definition}(Biane Voiculescu, \cite{BV})
The free quadratic Wasserstein distance is defined as:
\begin{align*}
    W_2((X_1,\ldots,X_n),(Y_1,\ldots,Y_n))
    =\inf \Bigl\{\lVert (X_i^{'} - Y_i^{'})_{1\leq i\leq n}\rVert_2 /
    (X_1^{'},\ldots,X_n^{'},Y_1^{'},\ldots,Y_n^{'}) \subset (M_3,\tau)\nonumber\\
    (X_1^{'},\ldots,X_n^{'})\simeq(X_1,\ldots,X_n),(Y_1^{'},\ldots,Y_n^{'})\simeq(Y_1,\ldots,Y_n)\Bigl\}
, \end{align*}
where $(M_3,\tau)$ is a $W^*$ tracial probability space,  with each $(X_i^{'},Y_i^{'})\in M_3$, and where $\simeq$ means the equality in distribution.
\end{definition}

\begin{theorem}($L^2$-transport inequalities)
\newline
For $k=2$, we have:
\begin{equation}
    W_2(X,S)\leq \max(\Sigma_2^*(X|V_1)(1-log(\Sigma_2^*(X|V_1
    )/\Sigma^*(X|V_1))),\Sigma_2^*(X|V_1)),
\end{equation}
And, for $k\geq 3$:
\begin{equation}
W_2(X,S)\leq \Sigma^*(X|V_1)^{1-\frac{1}{k-1}}\Sigma^*_k(X|V_1)^{\frac{1}{k-1}}
\end{equation}
\end{theorem}
\begin{proof}
By the Theorem 26 in \cite{Dab10}, which gives an estimation of the free quadratic Wasserstein distance in terms of the integrated relative free Fisher information along the Ornstein-Uhlenbeck flow, we have:
\begin{eqnarray}
    W_2(X,S)&\leq &\int_0^{\infty}\Phi^*(X(t)|V_{1})^{\frac{1}{2}}dt \nonumber\\
    & \leq & \Sigma^*(X|V_1)\int_0^{t}\frac{e^{-2s}}{\sqrt{1-e^{-2s}}}+\Sigma_2^*(X|V_1)\int_t^{\infty}\frac{e^{-3s}}{{1-e^{-2s}}}ds\nonumber\\
    &\leq & \sqrt{1-e^{-2t}}\Sigma^*(X|V_1)-\frac{1}{2}\log(1-e^{-2t})\Sigma_2^*(X|V_1)\nonumber
\end{eqnarray}
Note that the passage from the second to the third line comes from the inequality $e^{-3s}\leq e^{-2s}$ for $s\geq 0$.
\newline
Then by choosing "$t$" such that: 
\begin{equation}
\sqrt{1-e^{-2t}}=\min(\Sigma_2^*(X|V_1)/\Sigma^*(X|V_1)),1)
\end{equation}
and checking the two possible cases which are $\Sigma_2^*(X|V_1)\leq \Sigma^*(X|V_1)$ or $\Sigma_2^*(X|V_1)\leq \Sigma^*(X|V_1)$, we finally obtain the result.
\bigbreak
For the second statement, we apply the same scheme and we used simple majorations $e^{-(k+1)s}\leq e^{-(k-1)s}$ for $s\geq 0$:
\begin{eqnarray}
 W_2(X,S)&\leq &\int_0^{\infty}\Phi^*(X(t)|V_1)^{\frac{1}{2}}dt \nonumber\\
    & \leq & \Sigma^*(X|V_1)\int_0^{t}\frac{e^{-2s}}{\sqrt{1-e^{-2s}}}+\Sigma_k^*(X|V_1)\int_t^{\infty}\frac{e^{-(k+1)s}}{({1-e^{-2s}})^{\frac{k}{2}}}ds\nonumber\\
    &\leq & \sqrt{1-e^{-2t}}\Sigma^*(X|V_1)+\Sigma_k^*(X|V_1)\frac{e^{-(k-1)t}}{k-1}\left(\frac{1}{1-e^{-2s}}\right)^{\frac{k-2}{2}}\nonumber\\
    &\leq & \sqrt{1-e^{-2t}}\Sigma^*(X|V_1)+\Sigma_k^*(X|V_1)\left(\frac{1}{1-e^{-2s}}\right)^{\frac{k-2}{2}}
\end{eqnarray}
Choosing "$t$", such that:
\begin{equation}
\sqrt{1-e^{-2t}}=\left(\Sigma_k^*(X|V_1)/\Sigma^*(X|V_1)\right)^{\frac{1}{k-1}},
\end{equation}
and injecting this value into the inequality gives the desired result.
\end{proof}
\qed

\section{Improved rate of convergence in the free CLT}
\begin{flushleft}
In this section we will provide some new bounds in the free entropic {\it Central limit theorem} under higher-order moments constraints by using our previous {\it HSI} inequalities. We will work under an isotropy condition, and we will prove the following version of the entropic CLT with respect to the standard normalized semicircular potential.
\end{flushleft}

\begin{flushleft}
For each $j=1,\ldots,N$, let's $(x_j^{(N)})_{N\geq 1}$ be a sequence of centered, freely independent, and identically distributed noncommutative random variables in some $W^*$-tracial probability space $(\mathcal{M},\tau)$. We also assume that for each $n\geq 1$, the tuple $X^{(N)}=(x_1^{(N)},\ldots,x_n^{(N)})_{N\geq 1}$ is isotropic.
\end{flushleft}
\begin{flushleft}
We then denote $(a_1^{(N)},\ldots,a_N^{(N)})_{N\in \mathbb{N}}$ an array of real numbers such that $\forall N\geq 0$:
\begin{equation}
 \sum_{j=1}^N (a_j^{(N)})^2=1
\end{equation}
and we set $Y^{(N)}_j$ (the weighted variant in the CLT, i.e to find the usual free $CLT$, take $a_i=\frac{1}{\sqrt{N}}$)
\end{flushleft}
\begin{equation}
Y^{(N)}=a_1^{(N)}x^{(1)}+\ldots+a_N^{(N)}x^{(N)}
\end{equation}
Finally $\sigma_{N,k}$ denote the {\it weighted moment of order "$2(k+1)$"}:
\begin{equation}
\sigma_{N,k}=\sum_{j=1}^N (a_j^{(N)})^{2(k+1)}
\end{equation}

\begin{theorem}
Assume in the previous definitions, that the tuple $X^{(1)}$ admits free Stein kernel up to order $k$ (and thus that all its mixed moments strictly less than "$k+1$" that agrees with those of a standard semicircular system).
Assuming also that $X^{(1)}$ have first and second-order conjugates variables.
Then, the sequence $Y^{(N)}$ satisfies the following version of the entropic {\it CLT}:

\begin{equation}
\lvert\chi^{*}(Y^{(N)}|V_1)-\chi^{*}(S|V_1)\rvert\leq \mathcal{O}(\sigma_{N,k}^{\frac{1}{k}}\Phi^*(X|V_1)^{\frac{k-1}{k}})
\end{equation}
\end{theorem}
\begin{proof}
Denote $A_{k,1},\ldots,A_{k,N}$, be respectively higher-order free Stein kernels of order $k$ with respect to the potential $V_{I_n}$, of  the sequence $X^{(1)},\ldots,X^{(N)}$. Now, without restriction, since the $X^{(i)}$'s are equal in distribution, we can assume that we have the same norm for the free Stein discrepancies of order "$k$".
\newline
It is easily seen by a straightforward induction that:
\begin{equation}
    S_k=\E_N\bigg(\sum_{j=1}^N a_j^{k+1}A_{k,j}\bigg)
\end{equation}
is a free Stein kernel of order "$k$" for $Y^{(N)}$ where we denote $\mathbb{E}_N$ has the orthogonal projection onto $L^2\bigg(T_{(k+1,n)}(W^*(Y^{(N)})^{\otimes(n+1)})\bigg)$.
\newline
Indeed at the order $1$, the statement is already well known.
\newline
Thus, suppose that the statement is true for some $k\geq 1$, then
\begin{eqnarray}
\langle Y^{(N)},P(Y^{(N)})\rangle_2-(\tau\otimes\tau)\circ Tr(([JP](X))^*)
&=&\sum_{i=1}^na_i^{(N)} \langle A_{1,i},a_i^{(N)}[JP](Y^{(N)})\rangle_{HS}\nonumber\\
&=&\sum_{i=1}^na_i^{(N)} \langle A_{k,i},(a_i^{(N)})^{k}[J^kP](Y^{(N)})\rangle_{2}\nonumber
\end{eqnarray}
which suffices to conclude the induction.
Therefore, \begin{equation}
\lVert S_k\rVert_{2}^2\leq \sigma_{N,k}\lVert A_{k,1}\rVert_{2}^2
\end{equation}
\newline

Then, the result follows directly from the monotonicity of the free Fisher information along the CLT, the freeness of the random variables and the free {\it HSI} inequality.
\end{proof}
\qed
\section{Open Questions}\label{op}
Firstly, it would be very interesting to to investigate the construction of free Stein kernels via moment maps in the multivariate setting, which seems far more challenging, even in the {\it pertubative regime}. This would be of great importance to obtain new rates in the quantitative version of the free central limit theorem. Indeed in the classical case, it allows Fathi in \cite{mm} to deduce new bounds on the rate of convergence in the multi-dimensional 
central limit theorem when the random variables are log-concave, with explicit
dependence on the dimension, and even match the sharp rate (with respect to the dimension) when the initial measure is uniformly log-concave. 
\bigbreak
The recent work of Fathi and Mikulincer \cite{FM} provides new insights on how to prove stability estimates for invariant measures of diffusions, which can be far from the gaussian. As mentioned before, the (classical) SDE's considered in this paper are generally irregular and not necessarily Lipschitz, but belong generally to some Sobolev spaces. The first authors to outpass these issues were Figalli in \cite{Figal} for stochastic differential equation, and Le-Bris and Lions \cite{DLyo} for Fokker-Plank equations with irregular coefficients. Unfortunately, in the free setting, almost nothing is known in a general context of free SDE's with non {\it Operator Lipschitz} coefficients. It would be of great interest to develop a free analog of these results, and also to investigate if a kind of Lusin-Lipschitz property, that is maximal estimates for Sobolev functions w.r.t the Lesbesgue measure, and recently extended by Ambrosio, Bru\'{e} and Trevisan \cite{Lusin} when the reference measure is log-concave, could hold in the noncommutative context (see e.g. section 2.2 in \cite{FM} for precise statements).
\bigbreak
Recently new quantities of interest, and related to the notion of free Stein Kernel and Stein discrepancies: {\it the free Stein Irregularity and Dimension} were introduced by Charlesworth and Nelson \cite{CN,CN2}. It turns out these quantities are closely linked with the the free entropy dimension and agree in some case, as for the the group algebra of a discrete group (which are also deeply linked to the $\ell^2$-Betti numbers of discrete groups, as defined by Atiyah in \cite{Atiy}, see e.g. section 5 in \cite{CN} for precise statements and explicit computations of the free Stein dimension for discrete groups). Moreover, and surprisingly the free Stein dimension is equal to the Murray-von-Neumann dimension of the closure of the domain of the adjoint of the noncommutative Jacobian and the free Stein dimension turns out to be an $*-$algebra invariant. Some open questions as described in the papers of Charlesworth and Nelson are still open and it would be of great interest to have a deeper understanding of the relations between both the free entropy dimension and the free Stein dimension. 
\newline
In this section, we won't inspect further this question, but in light of our previous definition, we will define the generalisations of these notions with respect to free higher-order Stein kernels. We leave to a future investigation the study of the properties of these generalized quantities, if of course existence if ensured.
\bigbreak
\begin{definition}
Given $B$ a subalgebra of $\mathcal{M}$ (tracial $W^*$-probability space) and $\Theta \in L^2(\mathcal{M}^n)$, we call $A \in 
T_{(k+1,n)}(L^2( B\langle X\rangle^{(k+1)}
))$ a free stein Kernel of order $k$ over $B$, if for every $P \in B\langle X\rangle^n $ the following condition is true :
\begin{equation}
     \langle \Theta,P(X)\rangle_{\tau}-(\tau\otimes\tau)\circ Tr(([JP](X))^*)=\langle A_k,[J^k P](X)\rangle_2
\end{equation}
\end{definition}
We then denote $\Sigma_k^*(X|\Theta :B)=\inf \lVert A_k \rVert_2$ the associated discrepancy of $X$ relative to $\Theta$ over $B$.

At this point we were not able to construct explicitly such higher order Stein Kernels, even if we think that it should be possible. Thus, the continuity of the map which associate to a potential its associated discrepancy of order $k$ over $\mathbb{C}$ in the is in our case impossible to obtain. 
\bigbreak
We define the free Stein dimension of order $k$ of $X = (x_1,\ldots,x_n) \in \mathcal{M}^n$, a tuple of operators such that for each $i = 1,\ldots,n, \exists j\in {1,\ldots, n}/ x_i^* = x_j$, and where we denoted $B$ as an arbitrary unital $*$-subalgebra of $\mathcal{M}$ as: 

\begin{equation}
\Sigma_{k}^{*}(X|B)=\inf\left\{\Sigma_k^*(X|\Theta:B) :\Theta \in L^2(\mathcal{M}^n)\right\}
\end{equation}
Then it would be interesting to investigate the following questions:
\begin{enumerate}
    \item The main probabilistic uses of these generalized free Stein discrepancy is to obtain better rate of convergence in {\it CLT's}. The idea of free Stein dimension was also motivated heuristically by Charlesworth and Nelson by the following idea: as smaller the free Stein irregulairy is, the closer is $X$ to have conjugate variables relatively to $B$. Thus one should expect that some kind of ordering holds:
    Does $\Sigma_k^*(X|B)$ a decreasing function of $k$ ? And in the case $B=\mathbb{C}$, what does it implies for the regularity of the distribution of $X$ ?
    \item For a discrete group $\Gamma$, and $x_1,\ldots,x_n \in \mathbb{C}[\Gamma]_{s.a}$ which generate the group algebra, is $\Sigma_k(x_1,\ldots,x_n|\mathbb{C})$  explicitly computable, and is it linked to the $\ell^2$-Betti numbers of the group?
    
\end{enumerate}

\end{document}